\documentclass{msc}
\usepackage{amssymb}
\usepackage{amsmath}
\usepackage{amsthm}
\usepackage{graphicx}
\usepackage{soul}
\usepackage{boondox-cal}
\usepackage{anyfontsize}
\usepackage{tikz}
\usetikzlibrary{arrows,calc,patterns,positioning,shapes}
\usetikzlibrary{decorations.pathmorphing}
\tikzset{
modal/.style={>=stealth',shorten >=1pt,shorten <=1pt,auto,
node distance=1.5cm,semithick},
world/.style={circle,draw,minimum size=1cm},
point/.style={circle,draw,fill=black,inner sep=0.5mm},
reflexive/.style={->,in=120,out=60,loop,looseness=#1},
reflexive/.default={5},
reflexive point/.style={->,in=135,out=45,loop,looseness=#1},
reflexive point/.default={25},
}
\newcommand{\weakrightarrow}{\rightarrowtriangle}

\newcommand{\BD}{\mathsf{BD}}

\newcommand{\Luktriangle}{\Luk_\triangle}
\newcommand{\HLuktriangle}{\mathcal{H}\Luk_\triangle}
\newcommand{\fourLukProb}{\four\Prob^{\Luk_\triangle}}
\newcommand{\HfourLukProb}{\mathcal{H}\fourLukProb}
\newcommand{\LfourLukProb}{\mathcal{L}_{\fourLukProb}}
\newcommand{\LBD}{\mathcal{L}_\mathsf{BD}}
\newcommand{\Luk}{{\mathchoice{\mbox{\sf\L}}{\mbox{\sf\L}}{\mbox{\sf\scriptsize\L}}{\mbox{\sf\tiny\L}}}}
\newcommand{\LLuk}{\mathcal{L}_{\Luk}}

\newcommand{\LLukProbsquare}{\mathcal{L}_{\Prob^{\Luk^2}_\triangle}}
\newcommand{\LukProbsquare}{\Prob^{\Luk^2}_\triangle}
\newcommand{\Prop}{\mathtt{Prop}}
\newcommand{\Luksquareorder}{\Luk^2_{(\triangle,\rightarrow)}}
\newcommand{\LLuksquareorder}{\mathcal{L}_{\Luk^2_{(\triangle,\rightarrow)}}}

\newcommand{\Ninvol}{{\sim_\mathsf{N}}}
\newcommand{\np}{\mathsf{NP}}
\newcommand{\conp}{\mathsf{coNP}}

\newcommand{\four}{\mathbf{4}}
\newcommand{\HLukProbsquare}{\mathcal{H}\LukProbsquare}
\newcommand{\true}{\mathbf{T}}
\newcommand{\both}{\mathbf{B}}
\newcommand{\neither}{\mathbf{N}}
\newcommand{\false}{\mathbf{F}}

\newcommand{\bel}{\mathtt{bel}}
\newcommand{\purebel}{\mathsf{b}}
\newcommand{\puredisbel}{\mathsf{d}}
\newcommand{\conf}{\mathsf{c}}
\newcommand{\uncert}{\mathsf{u}}
\newcommand{\purebelmod}{\mathsf{Bl}}

\newcommand{\puredisbelmod}{\mathsf{Db}}
\newcommand{\conflmod}{\mathsf{Cf}}
\newcommand{\uncertmod}{\mathsf{Uc}}
\newcommand{\HLuksquareorder}{\mathcal{H}\Luksquareorder}
\newcommand{\TLuksquareorder}{\mathcal{T}\!\!\left(\Luksquareorder\right)}
\newcommand{\NLuk}{\mathsf{N}\Luk}
\newcommand{\TNLuk}{\mathcal{T}\!\!\left(\NLuk\right)}
\newcommand{\LNLuk}{\mathcal{L}_{\NLuk}}
\newcommand{\Prob}{\mathsf{Pr}}
\newcommand{\Bel}{\mathsf{B}}
\newcommand{\Pl}{\mathsf{Pl}}
\newcommand{\Lit}{\mathsf{Lit}}
\newcommand{\Sf}{\mathtt{Sf}}
\newcommand{\BelLuksquareorder}{\mathsf{Bel}^{\Luk^2}_\triangle}
\newcommand{\LBelLuksquareorder}{\mathcal{L}_{\BelLuksquareorder}}
\newcommand{\BelLuksquareNelson}{\mathsf{Bel}^{\NLuk}}
\newcommand{\LBelLuksquareNelson}{\mathcal{L}_{\BelLuksquareNelson}}
\newcommand{\NNF}{\mathsf{NNF}}
\newcommand{\cl}{\mathsf{cl}}
\newcommand{\canonicalBDequiv}{\mathfrak{M}^\mathsf{\BD\mathbf{{S5}}}}

\newcommand{\ProbLukordermodal}{\Prob^{(\triangle,\rightarrow)}_{\mathbf{S5}}}
\newcommand{\ProbLukNelsonmodal}{\Prob^{\NLuk}_{\mathbf{S5}}}
\newcommand{\LProbLukordermodal}{\mathcal{L}_{\ProbLukordermodal}}
\newcommand{\LProbLukNelsonmodal}{\mathcal{L}_{\ProbLukNelsonmodal}}

\newcommand{\pl}{\mathtt{pl}}

\newcommand{\Cmsf}{{\mathsf{C}}}

\newcommand{\Smsf}{{\mathsf{S}}}

\newcommand{\smsf}{\mathsf{s}}

\newcommand{\Bmc}{\mathcal{B}}

\newcommand{\Omc}{{\mathcal{O}}}

\newcommand{\Zmc}{{\mathcal{Z}}}

\newcommand{\lmc}{{\mathcal{l}}}
\newcommand{\tmc}{{\mathcal{t}}}
\usepackage{mathtools}
\usepackage{stmaryrd}
\usepackage[all]{xy}
\usepackage[makeroom]{cancel}
\usepackage{natbib}
\begin{document}
\allowdisplaybreaks
\lefttitle{B\'ilkov\'a, Frittella, Kozhemiachenko, and Majer}
\righttitle{Two-layered logics for probabilities and belief functions over Belnap--Dunn logic}
\papertitle{Article}
\jnlPage{1}{00}
\jnlDoiYr{2019}
\doival{10.1017/xxxxx}
\title{Two-layered logics for probabilities and belief functions over Belnap--Dunn logic}
\author{Marta B\'ilkov\'a}
\affiliation{Institute of Computer Science, The Czech Academy of Sciences, Prague\\
\email{bilkova@cs.cas.cz}}
\author{Sabine Frittella}
\affiliation{INSA CVL, Universit\'{e} d'Orl\'{e}ans, LIFO, UR 4022, Bourges, France\\
%INSA CVL, Université d’Orléans, LIFO, UR 4022, Blois, France
\email{sabine.frittella@insa-cvl.fr}}
\author{Daniil Kozhemiachenko}
\affiliation{Univ.\ Bordeaux, CNRS, Bordeaux INP, LaBRI, UMR 5800\\\email{daniil.kozhemiachenko@u-bordeaux.fr}}
\author{ Ondrej Majer}
\affiliation{Institute of Philosophy, The Czech Academy of Sciences, Prague\\\email{majer@flu.cas.cz}}

\history{(Received xx xxx xxx; revised xx xxx xxx; accepted xx xxx xxx)}
%\received{20 March 1995; revised 30 September 1998}

\begin{abstract}
This paper is an extended version of~\cite{BilkovaFrittellaKozhemiachenkoMajer2023WoLLIC}. We discuss two-layered logics formalising reasoning with probabilities and belief functions that combine the Łukasiewicz $[0,1]$-valued logic with Baaz $\triangle$ operator and the Belnap--Dunn logic.

We consider two probabilistic logics --- $\LukProbsquare$ (introduced by~\cite{BilkovaFrittellaKozhemiachenkoMajerNazari2023APAL}) and $\fourLukProb$ (from~\cite{BilkovaFrittellaKozhemiachenkoMajer2023WoLLIC}) --- that present two perspectives on the probabilities in the Belnap--Dunn logic. In $\LukProbsquare$, every event $\phi$ has independent positive and negative measures that denote the likelihoods of $\phi$~and $\neg\phi$, respectively. In $\fourLukProb$, the measures of the events are treated as partitions of the sample into four exhaustive and mutually exclusive parts corresponding to pure belief, pure disbelief, conflict and uncertainty of an agent in $\phi$.

In addition to that, we discuss two logics for the paraconsistent reasoning with belief and plausibility functions from~\cite{BilkovaFrittellaKozhemiachenkoMajerNazari2023APAL} --- $\BelLuksquareorder$ and $\BelLuksquareNelson$. Both these logics equip events with two measures (positive and negative) with their main difference being that in $\BelLuksquareorder$, the negative measure of $\phi$ is defined as the \emph{belief in $\neg\phi$} while in $\BelLuksquareNelson$, it is treated independently as \emph{the plausibility of $\neg\phi$}.

We provide a sound and complete Hilbert-style axiomatisation of $\fourLukProb$ and establish faithful translations between it and $\LukProbsquare$. We also show that the validity problem in all the logics is $\conp$-complete.
\end{abstract}


\maketitle
\section{Introduction\label{sec:introduction}}
Classical probability and Dempster--Shafer theories study probability measures and belief functions. These are maps from the set of events of a sample space $W$ (i.e., from $2^W$) to $[0,1]$ that are monotone w.r.t.\ $\subseteq$ with additional conditions. Probability measures satisfy the (finite or countable) additivity condition
\begin{align*}
\mu\left(\bigcup\limits_{i\in I}E_i\right)&=\sum\limits_{i\in I}\mu(E_i)\tag{$I\subseteq2^W,\forall i,j\in I:i\neq j\Rightarrow E_i\cap E_j=\varnothing$}
\end{align*}
and belief functions its weaker version (\emph{total monotonicity} in the terminology of~\cite{Zhou2013})
\begin{align*}
\bel(W)&=1&\bel\left(\bigcup\limits_{1\leq i\leq k}E_i\right)&\geq\sum\limits_{\scriptsize{\begin{matrix}J\subseteq\{1,\ldots,k\}\\J\neq\varnothing\end{matrix}}}(-1)^{|J|+1}\cdot\bel\left(\bigcap\limits_{j\in J}E_j\right)&
\bel(\varnothing)&=0
\end{align*}

Above, the disjointness of $E_i$ and $E_j$ can be understood as their incompatibility. Most importantly, if a~propositional formula $\phi$ is associated with an event $\|\phi\|$ (and interpreted as a statement about it), then
\[\mu(\|\phi\wedge\neg\phi\|)=\bel(\|\phi\wedge\neg\phi\|)=0\]
(since $\|\phi\|$ and $\|\neg\phi\|$ are incompatible) and $\|\phi\vee\neg\phi\|$ exhausts the entire sample space, whence
\[\mu(\|\phi\vee\neg\phi\|)=\bel(\|\phi\vee\neg\phi\|)=1\]

\emph{Paraconsistent} uncertainty theory, on the other hand, assumes that the measure of an event represents not the likelihood of it happening but an agent's certainty therein which they infer from the information given by the sources. In this respect, it is close to the classical Dempster--Shafer theory that can also be considered as a generalisation of the subjective probability theory.

The main difference between the classical and paraconsistent approaches is the treatment of negation. Dempster--Shafer theory usually deals with contradictions \emph{between different sources}. However, it is reasonable to assume that even a~\emph{single source} can give contradictory information (or give no information at all): for instance, a~witness in court can provide a~contradictory testimony; likewise, a~newspaper article can fail to mention some event at all, without explicitly denying or confirming it. Thus, a ‘contradictory’ event $\|\phi\wedge\neg\phi\|$ can have a positive probability or belief assignment and $\|\phi\vee\neg\phi\|$ does not necessarily exhaust the sample space. Thus, a~logic describing events should allow them to be both true and false (if the source gives contradictory information) or neither true nor false (when the source does not give information). Formally, this means that $\neg$ does not correspond to the complement in the sample space.
\subsection{Belnap--Dunn logic\label{ssec:BD}}
As one can see from the previous paragraph, we need a very special kind of negation: the one that allows for true contradictions (and thus, rejects the principle of explosion), and, additionally, invalidates the law of excluded middle. Thus, the following principles are no longer valid
\begin{align*}
\mathsf{EFQ}:p\wedge\neg p&\models q&\mathsf{LEM}:p\models q\vee\neg q
\end{align*}
The logics that lack $\mathsf{EFQ}$ are called \emph{paraconsistent}, those that do not have $\mathsf{LEM}$ are \emph{paracomplete}, and those that fail both principles are \emph{paradefinite} or \emph{paranormal} (cf., e.g.,~\cite{ArieliAvron2017} for the terminology).

The simplest paradefinite logic to represent reasoning about information provided by sources is the Belnap--Dunn logic ($\BD$) from~\cite{Dunn1976} and~\cite{Belnap1977fourvalued,Belnap2019}. Originally, $\BD$ was presented as a four-valued propositional logic in the $\{\neg,\wedge,\vee\}$ language. The values (which we will henceforth call \emph{Belnapian values}) represent different accounts a source can give regarding a~statement~$\phi$:
\begin{itemize}
\item $\true$ stands for ‘the source only says that $\phi$ is true’;
\item $\false$ stands for ‘the source only says that $\phi$ is false’;
\item $\both$ stands for ‘the source says both that $\phi$ is false and that $\phi$ is true’;
\item $\neither$ stands for ‘the source does not say that $\phi$ is false nor that it is true’.
\end{itemize}
The interpretation of the truth values allows for reformulating $\BD$ semantics in terms of \emph{two classical but independent valuations}. Namely,
\begin{center}
\begin{tabular}{c|c|c}
&\textbf{is true when}&\textbf{is false when}\\\hline
$\neg\phi$&$\phi$ is false&$\phi$ is true\\
$\phi_1\wedge\phi_2$&$\phi_1$ and $\phi_2$ are true&$\phi_1$ is false or $\phi_2$ is false\\
$\phi_1\vee\phi_2$&$\phi_1$ is true or $\phi_2$ is true&$\phi_1$ and $\phi_2$ are false
\end{tabular}
\end{center}
It is easy to see that there are no universally true nor universally false formulas in $\BD$. Thus, $\BD$ satisfies the desiderata outlined above. Moreover, even if we define $\phi\supset\chi$ as $\neg\phi\vee\chi$, it is clear that neither Modus Ponens, nor the Deduction theorem will hold for $\supset$. I.e., $\BD$ lacks the (definable) implication (cf.~\cite[\S5.1]{OmoriWansing2017} for a detailed discussion of the implication in $\BD$). This, however, is not problematic since we are going to use $\BD$ only to represent events and conditional statements which are usually formalised with an implication do not correspond to descriptions of events.
\subsection{Probabilities and belief functions in $\BD$}
The original interpretation of the Belnapian truth values that we gave above is presented in terms of the information one has. In this approach, however, the information is assumed to be crisp. Theories of uncertainty over $\BD$ were introduced to formalise situations where one has access to graded information. For instance, the first source could tell that $p$ is true with the likelihood $0.4$ and the second that $p$ is false with probability $0.7$. If one follows $\BD$ and treats positive and negative evidence independently, one needs a non-classical notion of uncertainty measures to represent this information.

To the best of our knowledge, the earliest formalisation of probability theory in terms of $\BD$ was provided by~\cite{Mares1997}. The formalisation is very similar to the one that we are going to use in this paper but bears one significant distinction. Namely, \emph{normalised} measures (i.e., those where $\mu(W)=1$ and $\mu(\varnothing)=0$) are used by~\cite{Mares1997}. This requirement, however, is superfluous since $\BD$ lacks universally (in)valid formulas.

Another formalisation is given by~\cite{Dunn2010}. Dunn proposes to divide the sample space into four exhaustive and mutually exclusive parts depending on the Belnapian value of $\phi$. An alternative approach was proposed by~\cite{KleinMajerRad2021}. There, the authors propose two equivalent interpretations based on the two formulations of semantics. The first option (which we call here \emph{$\pm$-probability}) is to give $\phi$ two \emph{independent probability measures}: the one determining the likelihood of $\phi$ to be true and the other the likelihood of $\phi$ to be false. The second option (\emph{$\four$-probabilities}) follows Dunn and divides the sample space according to whether $\phi$ has value $\true$, $\both$, $\neither$, or $\false$ in a given state. Note that in both cases, the probabilities are interpreted \emph{subjectively}.

The main difference between the approaches of~\cite{Dunn2010} and~\cite{KleinMajerRad2021} is that in the former, the probability of $\phi\wedge\phi'$ is entirely determined by those of $\phi$ and $\phi'$ which makes it compositional. On the other hand, the paraconsistent probabilities proposed by~\cite{KleinMajerRad2021} are not compositional w.r.t.\ conjunction. In this paper, we choose the latter approach since it can be argued (cf.~\cite{Dubois2008} for the details) that belief is not compositional when it comes to contradictory information.

A similar approach to paraconsistent probabilities was proposed by, e.g.,~\cite{Bueno-SolerCarnielli2016} and~\cite{RodriguesBueno-SolerCarnielli2021}. There, probabilities are defined over an extension of $\BD$ with classicality and non-classicality operators. It is worth mentioning that the proposed axioms of probability are very close to those from~\cite{KleinMajerRad2021}: e.g., both allow measures $\pi$ s.t.\ $\pi(\phi)+\pi(\neg\phi)<1$ (if the information regarding $\phi$ is incomplete) or $\pi(\phi)+\pi(\neg\phi)>1$ (when the information is contradictory).

Belief functions over $\BD$ were first defined by~\cite{Zhou2013}. There, they were presented on the ordered sets of states. Each formula $\phi$ in this approach corresponds to \emph{two} sets of states: $|\phi|^+$ (states where $\phi$ has value $\true$ or $\both$) and $|\phi|^-$ (states where it is evaluated at $\both$ or $\false$). Moreover, a logic formalising reasoning with belief functions was presented. A~similar treatment of (non-normalised or \emph{general} in the terminology of the paper) belief functions in $\BD$ was given by~\cite{BilkovaFrittellaKozhemiachenkoMajerNazari2023APAL}. The main difference between the two treatments of belief functions was that \cite{BilkovaFrittellaKozhemiachenkoMajerNazari2023APAL} considered two options for interpreting $\bel(|\phi|^-)$: the first one was to treat it as the belief of $\neg\phi$, and the second one --- as \emph{the plausibility} of $\neg\phi$. Another distinction was in the formalisation: \cite{Zhou2013} constructs a logic for reasoning about belief functions following~\cite{FaginHalpernMegiddo1990}: i.e., incorporating the arithmetical operations and inequalities containing them into the language. \cite{BilkovaFrittellaKozhemiachenkoMajerNazari2023APAL} utilised a different approach: instead of using arithmetic inequalities, the reasoning about belief functions is conducted in a~paraconsistent expansion of the Łukasiewicz logic $\Luk$ capable of expressing arithmetic operations on~$[0,1]$.
\subsection{Two-layered logics for uncertainty\label{ssec:2layereduncertainty}}
Reasoning about uncertainty can be formalised via modal logics where the modality is understood as a measure of an event. The concrete semantics of the modality can be defined in two ways. First, using a modal language with Kripke semantics where the measure is defined on the set of states as done by, e.g.,~\cite{Gardenfors1975,DelgrandeRenne2015,DelgrandeRenneSack2019} for qualitative probabilities, by~\cite{DautovicDoderOgnjanovic2021} for the quantitative ones, and by~\cite{RodriguezTuytEstevaGodo2022} for the possibility and necessity measures. Second, employing a two-layered formalism (cf.~\cite{FaginHalpernMegiddo1990,FaginHalpern1991ComputationalIntelligence}, \cite{HajekGodoEsteva1995}, \cite{BaldiCintulaNoguera2020}, and~\cite{BilkovaFrittellaKozhemiachenkoMajerNazari2023APAL,BilkovaFrittellaKozhemiachenkoMajer2023IJAR,BilkovaFrittellaKozhemiachenkoMajer2023WoLLIC} for examples). There, the logic is split into two levels: the inner layer describes events, and the outer layer describes the reasoning with the measure defined on events. The measure is a \emph{non-nesting} modality $\mathtt{M}$, and the outer-layer formulas are built from \emph{‘modal atoms’} --- formulas of the form $\mathtt{M}\phi$ with $\phi$ being an inner-layer formula. The outer-layer formulas are then equipped with the semantics of a~fuzzy logic that permits necessary operations (e.g., Łukasiewicz for the quantitative reasoning and G\"{o}del for the qualitative).

An alternative to the two-layered logics is to use the language of linear inequalities to reason about measures of events. This is done by~\cite{FaginHalpernMegiddo1990} for the classical reasoning about probabilities and by~\cite{Zhou2013} for the reasoning with the belief functions over $\BD$. In both cases, it is established that the logics with inequalities and the two-layered logics are equivalent~--- cf.~\cite{BaldiCintulaNoguera2020} for the case of classical probabilities and~\cite{BilkovaFrittellaKozhemiachenkoMajerNazari2023APAL} for the reasoning with belief functions and probabilities over $\BD$.

In this paper, we study reasoning about uncertainty in a~paraconsistent framework. For this, we choose two-layered logics. First, they are more modular than the usual Kripke semantics: as long as the logic of the event description is chosen, we can define different measures on top of it using different outer-layer logics. Second, two-layered logics provide a~uniform way to prove completeness. This is done by translating the axioms of a~given measure into formulas of the outer-layer logic (cf.~\cite{CintulaNoguera2014} for more details). Third, the techniques that are used to establish the decidability of the outer-layer logic can be applied to the decidability proof of the two-layered logic. Finally, even though the traditional Kripke semantics is more expressive than two-layered logics, this expressivity is not necessary in many contexts. Indeed, people rarely say something like ‘it is probable that it is probable that $\phi$’. Moreover, it is considerably more difficult to motivate the assignment of truth values in the nesting case, in particular, when the same measure is applied both to a~propositional and modalised formula as in, e.g., $\mathtt{M}(p\wedge\mathtt{M}q)$.

We will also be formalising the \emph{quantitative} reasoning. I.e., we assume that the agents can assign numerical values to their certainty in a given proposition or say something like ‘I think that rain is twice more likely than snow’. Thus, we need a logic that can express the paraconsistent counterparts of the additivity condition as well as basic arithmetic operations. We choose the Łukasiewicz logic ($\Luk$) for the outer layer since it can define (truncated) addition and subtraction on $[0,1]$.

We will focus on four logics. Two probabilistic ones: $\LukProbsquare$ (the logic of $\pm$-probabilities), $\fourLukProb$ (the logic of $\four$-probabilities); and two logics for belief and plausibility functions over $\BD$ introduced by~\cite{BilkovaFrittellaKozhemiachenkoMajerNazari2023APAL} --- $\BelLuksquareorder$ where belief and plausibility are defined via one another, and $\BelLuksquareNelson$ where belief and plausibility are assumed to be independent.
\subsection{Contributions and plan of the paper\label{ssec:plan}}
This paper extends an earlier conference submission by~\cite{BilkovaFrittellaKozhemiachenkoMajer2023WoLLIC}. We provide omitted details for the proofs of $\conp$-completeness of $\LukProbsquare$ and $\fourLukProb$ and of the finite strong completeness proof of $\HfourLukProb$ --- the Hilbert-style axiomatisation of $\fourLukProb$. The novel contribution of this manuscript is the proof of the $\conp$-completeness of $\BelLuksquareorder$ and $\BelLuksquareNelson$. We obtain this result by establishing a~correspondence between these logics and $\ProbLukordermodal$ and $\ProbLukNelsonmodal$ --- two logics for reasoning about probabilities of \emph{modal} $\BD$ formulas that we introduce in the present paper. To apply the technique from~\cite{HajekTulipani2001}, we establish a~version of the small model property for the canonical models of $\ProbLukordermodal$ and $\ProbLukNelsonmodal$. Thus, we continue the study of uncertainty via paraconsistent logics proposed by~\cite{BilkovaFrittellaMajerNazari2020} and carried on by~\cite{BilkovaFrittellaKozhemiachenkoMajerNazari2023APAL}, \cite{BilkovaFrittellaKozhemiachenkoMajer2023IJAR}, and~\cite{BilkovaFrittellaKozhemiachenkoMajerManoorkar2023}. The overarching goal of the project is to study logics that can express the following properties of beliefs.
\begin{enumerate}
\item Given two statements $\phi$ and $\chi$, one can be more certain in $\phi$ than in $\chi$ but still, neither believe in $\phi$ \emph{completely} nor consider $\chi$ completely impossible.
\item Given two trusted sources, one can still prefer one source to the other.
\item One can believe in a~contradiction but still not believe in something else.
\item Given two statements, it is possible that one cannot always compare their degrees of certainty in them (if, e.g., these statements have no common content).
\end{enumerate}

The remainder of the text is structured as follows. In Section~\ref{sec:BDprobabilities}, we recall two approaches to probabilities over $\BD$ from~\cite{KleinMajerRad2021} --- $\pm$-probabilities and $\four$-probabilities. In Section~\ref{sec:twolayeredprobabilities}, we provide the semantics and axiomatisations of $\LukProbsquare$ and $\fourLukProb$ --- two-layered logics for probabilities and establish their $\np$-completeness. In Section~\ref{sec:BDbelief}, we recall two treatments of belief functions over $\BD$ that were presented by~\cite{BilkovaFrittellaKozhemiachenkoMajerNazari2023APAL}. In Section~\ref{sec:twolayeredbelief}, we discuss the two-layered logics for belief functions ($\BelLuksquareorder$ and $\BelLuksquareNelson$), establish their connections with modal probabilistic logics $\ProbLukordermodal$ and $\ProbLukNelsonmodal$, and provide a~complexity evaluation of their validity problems. Finally, we summarise our results and set goals for future research in Section~\ref{sec:conclusion}.
\section{Probabilities over $\BD$\label{sec:BDprobabilities}}
In the previous section, we gave an informal presentation of $\BD$ as a four-valued logic. Since we are going to use it to describe events, we will formulate its semantics in terms of sets of states. The language of $\BD$ is given by the following grammar (with $\Prop$ being a countable set of propositional variables).
\[\LBD\ni\phi\coloneqq p\in\Prop\mid\neg\phi\mid(\phi\wedge\phi)\mid(\phi\vee\phi)\]
\begin{convention}[Notation]
In what follows, $\Prop(\phi)$ denotes the set of variables occurring in $\phi$ and $\Lit(\phi)$ stands for the set of \emph{literals} (i.e., variables or their negations) occurring in $\phi$. Moreover, $\Sf(\phi)$ is the set of all subformulas of $\phi$. The \emph{length} (i.e., the number of symbols) of $\phi$ is denoted with $\lmc(\phi)$.

We are also going to use two kinds of formulas: the single- and the two-layered ones. To make the differentiation between them simpler, we use Greek letters from the end of the alphabet ($\phi$, $\chi$, $\psi$, etc.) to designate the first kind and the letters from the beginning of the alphabet ($\alpha$, $\beta$, $\gamma$, \ldots) for the second kind. We use $v$ (with indices) to stand for the valuations of single-layered formulas and $e$ (with indices) for the two-layered formulas.

Finally, we will use angular brackets $\langle\ldots\rangle$ for tuples that designate \emph{models} of logics and \emph{pairs of valuations} and round brackets $(\ldots)$ for \emph{pairs of numbers} that are values of formulas. 
\end{convention}
\begin{definition}[Set semantics of $\BD$]\label{def:BDframesemantics}
Let $\phi,\phi'\in\LBD$, $W\neq\varnothing$, and $v^+,v^-:\Prop\to2^W$. For a~model $\mathfrak{M}=\langle W,v^+,v^-\rangle$, we define notions of $w\vDash^+\phi$ and $w\vDash^-\phi$ for $w\in W$ as follows.
\begin{align*}
w\vDash^+p&\text{ iff }w\in v^+(p)&w\vDash^-p&\text{ iff }w\in v^-(p)\\
w\vDash^+\neg\phi&\text{ iff }w\vDash^-\phi&w\vDash^-\neg\phi&\text{ iff }w\vDash^+\phi\\
w\vDash^+\phi\wedge\phi'&\text{ iff }w\vDash^+\phi\text{ and }w\vDash^+\phi'&w\vDash^-\phi\wedge\phi'&\text{ iff }w\vDash^-\phi\text{ or }w\vDash^-\phi'\\
w\vDash^+\phi\vee\phi'&\text{ iff }w\vDash^+\phi\text{ or }w\vDash^+\phi'&w\vDash^-\phi\vee\phi'&\text{ iff }w\vDash^-\phi\text{ and }w\vDash^-\phi'
\end{align*}
Given a~model $\mathfrak{M}$, we denote the positive and negative extensions of a~formula as follows:
\begin{align*}
|\phi|^+_\mathfrak{M}&\coloneqq\{w\in W\mid w\vDash^+\phi\}&|\phi|^-_\mathfrak{M}&\coloneqq\{w\in W\mid w\vDash^-\phi\}.
\end{align*}
A~sequent $\phi\vdash\chi$ is \emph{satisfied on $\mathfrak{M}=\langle W,v^+,v^-\rangle$} (denoted, $\mathfrak{M}\models[\phi\vdash\chi]$) iff $|\phi|^+\subseteq|\chi|^+$. A~sequent $\phi\vdash\chi$ is \emph{$\BD$-valid} ($\phi\models_\BD\chi$) iff it is satisfied on every model. In this case, we will say that $\phi$ \emph{entails}~$\chi$ \emph{in $\BD$}.
\end{definition}
\begin{remark}\label{rem:NNF}
One can see that every $\phi\in\LBD$ can be turned into its \emph{negation normal form} $\NNF(\phi)$ where $\neg$ is applied to variables only. This can be done via the following transformations:
\begin{align*}
\neg\neg p&\rightsquigarrow p&\neg(\chi\wedge\psi)&\rightsquigarrow\neg\chi\vee\neg\psi&\neg(\chi\vee\psi)&\rightsquigarrow\neg\chi\wedge\neg\psi
\end{align*}
Moreover, it is clear that $|\phi|^+_\mathfrak{M}=|\NNF(\phi)|^+_\mathfrak{M}$ and $|\phi|^-_\mathfrak{M}=|\NNF(\phi)|^-_\mathfrak{M}$ in every model $\mathfrak{M}$ and that $\lmc(\NNF(\phi))=\Omc(\lmc(\phi))$.
\end{remark}

Note that the semantics in~Definition~\ref{def:BDframesemantics} is a~formalisation of the truth and falsity conditions of $\{\neg,\wedge,\vee\}$-formulas we saw in the table in Section~\ref{ssec:BD}. We can now use it to define probabilities on the models. We adapt the definitions from~\cite{KleinMajerRad2021}.
\begin{definition}[$\BD$-models with $\pm$-probabilities]\label{def:measuredmodel}
A \emph{$\BD$-model with a $\pm$-probability} is a tuple $\mathfrak{M}_\pm=\langle\mathfrak{M},\mu\rangle$ with $\mathfrak{M}$ being a $\BD$-model and $\mu:2^W\rightarrow[0,1]$ satisfying:
\begin{description}
\item[$\mathbf{mon}$:] if $X\subseteq Y$, then $\mu(X)\leq\mu(Y)$;
\item[$\mathbf{neg}$:] $\mu(|\phi|^-_\mathfrak{M})=\mu(|\neg\phi|^+_\mathfrak{M})$;
\item[$\mathbf{ex}$:] $\mu(|\phi\vee\chi|^+_\mathfrak{M})=\mu(|\phi|^+_\mathfrak{M})+\mu(|\chi|^+_\mathfrak{M})-\mu(|\phi\wedge\chi|^+_\mathfrak{M})$.
\end{description}
\end{definition}
To facilitate the presentation of the four-valued probabilities defined over $\BD$-models, we introduce additional extensions of $\phi$ defined via $|\phi|^+$ and $|\phi|^-$.
\begin{convention}\label{conv:4measures}
Let $\mathfrak{M}=\langle W,v^+,v^-\rangle$ be a $\BD$-model, $\phi\in\LBD$. We set
\begin{align*}
|\phi|^\purebel_\mathfrak{M}=&|\phi|^+_\mathfrak{M}\setminus|\phi|^-_\mathfrak{M}& |\phi|^\puredisbel_\mathfrak{M}=&|\phi|^-_\mathfrak{M}\setminus|\phi|^+_\mathfrak{M}&
|\phi|^\conf=&|\phi|^+\cap|\phi|^-_\mathfrak{M}& |\phi|^\uncert_\mathfrak{M}=&W\setminus(|\phi|^+_\mathfrak{M}\cup|\phi|^-_\mathfrak{M})
\end{align*}
We call these extensions, respectively, \emph{pure belief}, \emph{pure disbelief}, \emph{conflict}, and \emph{uncertainty in $\phi$}, following~\cite{KleinMajerRad2021}.
\end{convention}
One can observe that these extensions correspond to the Belnapian values of $\phi$ (recall Section~\ref{ssec:BD}):
\begin{itemize}
\item \emph{pure belief extension} of $\phi$ is the set of states where $\phi$ has value $\true$ (i.e., \emph{exactly true}, in other words, true and not false);
\item \emph{pure disbelief extension} of $\phi$ is the set of states where $\phi$ has value $\false$ (i.e., \emph{exactly false}, in other words, false and not true);
\item \emph{conflict extension} of $\phi$ is the set of states where $\phi$ has value $\both$ (i.e., \emph{both true and false});
\item \emph{uncertainty extension} of $\phi$ is the set of states where $\phi$ has value $\neither$ (i.e., \emph{neither true nor false}).
\end{itemize}
\begin{definition}[$\BD$-models with $\four$-probabilities]\label{def:4model}
A \emph{$\BD$-model with a $\four$-probability} is a tuple $\mathfrak{M}_\four=\langle\mathfrak{M},\mu_\four\rangle$ with $\mathfrak{M}$ being a $\BD$-model and $\mu_\four:2^W\rightarrow[0,1]$ satisfying:
\begin{description}
\item[$\mathbf{part}$:] $\mu_\four(|\phi|^\purebel_\mathfrak{M})+\mu_\four(|\phi|^\puredisbel_\mathfrak{M})+\mu_\four(|\phi|^\uncert_\mathfrak{M})+\mu_\four(|\phi|^\conf_\mathfrak{M})=1$;
\item[$\mathbf{neg}$:] $\mu_\four(|\neg\phi|^\purebel_\mathfrak{M})=\mu_\four(|\phi|^\puredisbel_\mathfrak{M})$, $\mu_\four(|\neg\phi|^\conf_\mathfrak{M})=\mu_\four(|\phi|^\conf_\mathfrak{M})$;
\item[$\mathbf{contr}$:] $\mu_\four(|\phi\wedge\neg\phi|^\purebel_\mathfrak{M})=0$, $\mu_\four(|\phi\wedge\neg\phi|^\conf_\mathfrak{M})=\mu_\four(|\phi|^\conf_\mathfrak{M})$;
\item[$\mathbf{BCmon}$:] if $\mathfrak{M}\models[\phi\vdash\chi]$, then $\mu_\four(|\phi|^\purebel_\mathfrak{M})+\mu_\four(|\phi|^\conf_\mathfrak{M})\leq\mu_\four(|\chi|^\purebel_\mathfrak{M})+\mu_\four(|\chi|^\conf_\mathfrak{M})$;
\item[$\mathbf{BCex}$:] $\mu_\four(|\phi|^\purebel)+\mu_\four(|\phi|^\conf)+\mu_\four(|\psi|^\purebel)+\mu_\four(|\psi|^\conf)=\mu_\four(|\phi\wedge\psi|^\purebel)+\mu_\four(|\phi\wedge\psi|^\conf)+\mu_\four(|\phi\vee\psi|^\purebel)+\mu_\four(|\phi\vee\psi|^\conf)$.
\end{description}
\end{definition}
\begin{convention}
We will further omit lower indices in $|\phi|^+_\mathfrak{M}$, $|\phi|^-_\mathfrak{M}$, $|\phi|^\purebel_\mathfrak{M}$, etc.\ and write $|\phi|^+$, $|\phi|^-$, $|\phi|^\purebel$, etc.\ when the model is clear from the context.

We will utilise the following naming convention:
\begin{itemize}
\item we use the term \emph{‘$\pm$-probability’} to stand for $\mu$ from Definition~\ref{def:measuredmodel};
\item we call $\mu_\four$ from Definition~\ref{def:4model} a~\emph{‘$\four$-probability’} or a~\emph{‘four-valued probability’}.
\end{itemize}
Recall that $\pm$-probabilities are referred to as ‘non-standard’ by~\cite{KleinMajerRad2021} and~\cite{BilkovaFrittellaKozhemiachenkoMajerNazari2023APAL}. As this term is too broad (four-valued probabilities are not ‘standard’ either), we use a~different designation.
\end{convention}

Let us quickly discuss the measures defined above. First, observe that $\mu(|\phi|^+)$ and $\mu(|\phi|^-)$ are independent from one another. Thus, $\mu$ gives two measures to each $\phi$, as desired. Second, recall~\cite[Theorems~2--3]{KleinMajerRad2021} that every $\four$-probability on a $\BD$-model induces a $\pm$-probability and vice versa. In the following sections, we will define two-layered logics for $\BD$-models with $\pm$- and $\four$-probabilities and show that they can be faithfully embedded into each other. It is instructive to note, moreover, that while $\pm$- and $\four$-probabilities can be simulated by the classical ones, they do behave in a very different way.
\begin{convention}[Notation in models]
Throughout the paper, we are going to give examples of various models. We use the following notation for values of variables in the states of a~given model.
\begin{center}
\begin{tabular}{c|c}
\textbf{notation}&\textbf{meaning}\\\hline
$w:p^+$&$w\vDash^+p$ and $w\nvDash^-p$\\
$w:p^-$&$w\vDash^-p$ and $w\nvDash^+p$\\
$w:p^\pm$&$w\vDash^+p$ and $w\vDash^-p$\\
$w:\xcancel{p}$&$w\nvDash^+p$ and $w\nvDash^-p$
\end{tabular}
\end{center}
\end{convention}
\begin{example}[Paraconsistent vs.\ classical probabilities]\label{example:+-nomeasure}
First of all, observe that a~$\pm$-probability \emph{can be uniform}. Indeed, for every $c\in[0,1]$ and every $\BD$-model $\langle W,v^+,v^-\rangle$, it is easy to check that the assignment $\forall X\subseteq W:\mu(X)=c$ \emph{is a~$\pm$-probability}.

Second, the general import-export condition
\begin{align}\label{equ:exportimport}
\mu(X\cup Y)&=\mu(X)+\mu(Y)-\mu(X\cap Y)
\tag{$\mathsf{IE}$}
\end{align}
that is true for the classical probabilities does not hold if $\mu$ is a~$\pm$-probability. For consider Fig.~\ref{fig:importexportcounterexample}. One can see that $\mu$ is monotone and that $\mu(|\phi|^+)=\mu(|\phi|^-)=0$ for every $\phi\in\LBD$ (whence, $\mu$~is a~$\pm$-probability). On the other hand, it is clear that $\mu(W)\neq\mu(\{u_1\})+\mu(\{u_2\})-\mu(\varnothing)$.

\begin{figure}
\centering
\[\xymatrix{u_1:\xcancel{p}&u_2:\xcancel{p}}\]
\vspace{1em}
\caption{A counterexample to~\eqref{equ:exportimport}: $\mu(\{u_1\})=\mu(\{u_2\})=\frac{1}{3}$, $\mu(W)=1$, and $\mu(\varnothing)=0$. All variables have the same values exemplified by $p$.}
\label{fig:importexportcounterexample}
\end{figure}

This, however, is not a~problem since for every $\BD$-model with a~$\pm$-pro\-ba\-bi\-li\-ty $\langle W,v^+,v^-,\mu\rangle$, there exists a~$\BD$-model $\langle W',v'^+,v'^-,\pi\rangle$ with a~\emph{classical} probability measure $\pi$ s.t.\ $\pi(|\phi|^+)=\mu(|\phi|^+)$~\cite[Theorem~4]{KleinMajerRad2021}.

First of all, for the case of the model shown in Fig.~\ref{fig:importexportcounterexample}, one can either define $\pi(\{u_1\})=\pi(\{u_2\})=\frac{1}{2}$, or add a new state $u_3$ where all variables are neither true nor false. For a~bit more refined example, consider Fig.~\ref{fig:classicaluniformassignment} and set $W=\{w_1,w_2\}$ and $W'=\{w'_1,w'_2,w'_3\}$. It is clear that for every $\phi\in\LBD$, $\mu(|\phi|^+)=\pi(|\phi|^+)=\frac{1}{2}$.

Likewise, $\mu_\four$ is not necessarily monotone w.r.t.\ $\subseteq$ (which is required of the classical probabilities) since not every subset of a~model is represented by an extension of an $\LBD$ formula. Again, it is not a~problem since for every $\BD$-model with $\four$-probability $\langle W,v^+,v^-,\mu_\four\rangle$, there exist a~$\BD$-model $\langle W',v'^+,v'^-,\pi\rangle$ with a~\emph{classical} probability measure $\pi$ s.t.\ $\pi(|\phi|^\mathsf{x})=\mu_\four(|\phi|^\mathsf{x})$ for $\mathsf{x}\in\{\purebel,\puredisbel,\conf,\uncert\}$~\cite[Theorem~5]{KleinMajerRad2021}.

Thus, we will further assume w.l.o.g.\ that $\mu$ and $\mu_\four$ are \emph{classical probability measures} on $W$.
\end{example}
\begin{figure}
\centering
\[\xymatrix{w_1:p^+&w_2:p^\pm&&w'_1:p^+&w'_2:p^\pm&w'_3:\xcancel{p}}\]
\vspace{1em}
\caption{The values of all variables coincide with the values of $p$ state-wise. $\mu(X)=\frac{1}{2}$ for every $X\subseteq W$; $\pi(\varnothing)=\pi(\{w'_1\})=0$, $\pi(W')=1$, $\pi(X')=\frac{1}{2}$ otherwise.}
\label{fig:classicaluniformassignment}
\end{figure}
\section{Logics for paraconsistent probabilities\label{sec:twolayeredprobabilities}}
In the \nameref{sec:introduction}, we saw that there could be two views on formalising paraconsistent probabilities (or uncertainty measures in general). The first option is to define probabilities in a~paraconsistent logic. This is done, e.g., by~\cite{Dunn2010}, \cite{Bueno-SolerCarnielli2016}, and~\cite{KleinMajerRad2021}. Another option is to employ a~two-layered framework where the outer layer is not explosive as \cite{FlaminioGodoUgolini2022} do. Our approach can be thought of as a~combination of these two. Not only are our probabilities defined over a paraconsistent logic, but the logic with which we reason about them is also non-explosive.

In this section, we provide two-layered logics that are (weakly) complete w.r.t.\ $\BD$-models with $\pm$- and $\four$-probabilities. Since conditions on measures contain arithmetic operations on $[0,1]$, we choose an expansion of Łukasiewicz logic, namely, Łukasiewicz logic with~$\triangle$ ($\Luktriangle$), for the outer layer. Furthermore, $\pm$-probabilities work with both positive and negative extensions of formulas, whence it is reasonable to use $\Luksquareorder$ --- a paraconsistent expansion of $\Luk$ (cf.~\cite{BilkovaFrittellaMajerNazari2020} and~\cite{BilkovaFrittellaKozhemiachenko2021TABLEAUX} for details) with two valuations --- $v_1$ (support of truth) and $v_2$ (support of falsity) --- on $[0,1]$.
\subsection{Languages, semantics, and equivalence}
Let us first recall the semantics of $\Luktriangle$ and $\Luksquareorder$. To facilitate the presentation, we begin with the definition of the standard $\Luktriangle$ algebra on $[0,1]$ that we will then use to define the semantics of both $\Luktriangle$ and $\Luksquareorder$.
\begin{definition}\label{def:Lukalgebra}
The standard $\Luktriangle$ algebra is a tuple $\langle[0,1],{\sim_\Luk},\triangle_\Luk,\wedge_\Luk,\vee_\Luk,\rightarrow_\Luk,\odot_\Luk,\oplus_\Luk,\ominus_\Luk\rangle$ with the operations are defined as follows.
\begin{align*}
{\sim_\Luk}a&\coloneqq1-a&\triangle_\Luk a&\coloneqq\begin{cases}1&\text{if }a=1\\0&\text{otherwise}\end{cases}
\end{align*}
\begin{align*}
a\wedge_\Luk b&\coloneqq\min(a,b)&a\vee_\Luk b&\coloneqq\max(a,b)&a\rightarrow_\Luk b&\coloneqq\min(1,1-a+b)\\
a\odot_\Luk b&\coloneqq\max(0,a+b-1)&a\oplus_\Luk b&\coloneqq\min(1,a+b)&a\ominus_\Luk b&\coloneqq\max(0,a-b)
\end{align*}
\end{definition}
\begin{definition}[$\Luktriangle$]\label{def:Lukasiewicz}
The language of $\Luktriangle$ is given via the following grammar
\[\LLuk\ni\phi\coloneqq p\in\Prop\mid{\sim}\phi\mid\triangle\phi\mid(\phi\wedge\phi)\mid(\phi\vee\phi)\mid(\phi\rightarrow\phi)\mid(\phi\odot\phi)\mid(\phi\oplus\phi)\mid(\phi\ominus\phi)\]
We will also write $\phi\leftrightarrow\chi$ as a shorthand for $(\phi\rightarrow\chi)\odot(\chi\rightarrow\phi)$.

A valuation is a map $v:\Prop\rightarrow[0,1]$ that is extended to the complex formulas as expected: $v(\phi\circ\chi)=v(\phi)\circ_\Luk v(\chi)$.

$\phi$ is \emph{$\Luktriangle$-valid} iff $v(\phi)=1$ for every~$v$. $\Gamma$ \emph{entails} $\chi$ (denoted $\Gamma\models_{\Luktriangle}\chi$) iff for every $v$ s.t.\ $v(\phi)=1$ for all $\phi\in\Gamma$, it holds that $v(\chi)=1$ as well.
\end{definition}
\begin{definition}[$\Luksquareorder$]\label{def:Luk2triangle}
The language is constructed using the following grammar.
\[\LLuksquareorder\ni\phi\coloneqq p\in\Prop\mid\neg\phi\mid{\sim}\phi\mid\triangle\phi\mid(\phi\rightarrow\phi)\]
The semantics is given by \emph{two} valuations $v_1$ (support of truth) and $v_2$ (support of falsity) $v_1,v_2:\Prop\rightarrow[0,1]$ that are extended as follows.
\begin{align*}
v_1(\neg\phi)&=v_2(\phi)&v_2(\neg\phi)&=v_1(\phi)\\
v_1({\sim}\phi)&={\sim_\Luk}v_1(\phi)&v_2({\sim}\phi)&={\sim_\Luk}v_2(\phi)\\
v_1(\triangle\phi)&=\triangle_\Luk v_1(\phi)&v_2(\triangle\phi)&={\sim_\Luk}\triangle_\Luk{\sim}_\Luk v_2(\phi)\\
v_1(\phi\rightarrow\chi)&=v_1(\phi)\rightarrow_\Luk v_1(\chi)&v_2(\phi\rightarrow\chi)&=v_2(\chi)\ominus_\Luk v_2(\phi)
\end{align*}
We say that $\phi$ is \emph{$\Luksquareorder$-valid} iff for every $v_1$ and $v_2$, it holds that $v_1(\phi)=1$ and $v_2(\phi)=0$.
\end{definition}
\begin{convention}
When there is no risk of confusion, we will use $v(\phi)=(x,y)$ to stand for $v_1(\phi)=x$ and $v_2(\phi)=y$.
\end{convention}
\begin{remark}\label{rem:smalllanguage}
Note that ${\sim}$ and $\rightarrow$ can define all other binary connectives in $\Luktriangle$ and $\Luksquareorder$.
\begin{align*}
\phi\vee\chi&\coloneqq(\phi\rightarrow\chi)\rightarrow\chi&\phi\wedge\chi&\coloneqq{\sim}({\sim}\phi\vee{\sim}\chi)&\phi\oplus\chi&\coloneqq{\sim}\phi\rightarrow\chi\\
\phi\odot\chi&\coloneqq{\sim}(\phi\rightarrow{\sim}\chi)&\phi\ominus\chi&\coloneqq\phi\odot{\sim}\chi&\phi\leftrightarrow\chi&\coloneqq(\phi\rightarrow\chi)\odot(\chi\rightarrow\phi)
\end{align*}
\end{remark}

We are now ready to present our two-layered logics for paraconsistent probabilities: $\LukProbsquare$ --- the logic of $\pm$-probabilities and $\fourLukProb$ --- the logic of $\four$-probabilities.
\begin{definition}[$\fourLukProb$: language and semantics]\label{def:4ProbLuk}
The language of $\fourLukProb$ is constructed via the following grammar:
\begin{align*}
\LfourLukProb\ni\alpha&\coloneqq\purebelmod\phi\mid\puredisbelmod\phi\mid\conflmod\phi\mid\uncertmod\phi\mid{\sim}\alpha\mid\triangle\alpha\mid(\alpha\rightarrow\alpha)\tag{$\phi\in\LBD$}
\end{align*}
A $\fourLukProb$-model is a tuple $\mathbb{M}=\langle\mathfrak{M},\mu_\four,e\rangle$ s.t.
\begin{itemize}
\item $\langle\mathfrak{M},\mu_\four\rangle$ is a $\BD$-model with $\four$-probability;
\item $e$ is a \emph{$\Luktriangle$ valuation induced by $\mu_\four$}, i.e.:
\begin{itemize}
\item $e(\purebelmod\phi)=\mu_\four(|\phi|^\purebel)$, $e(\puredisbelmod\phi)=\mu_\four(|\phi|^\puredisbel)$, $e(\conflmod\phi)=\mu_\four(|\phi|^\conf)$, $e(\uncertmod\phi)=\mu_\four(|\phi|^\uncert)$;
\item values of complex $\LfourLukProb$-formulas are computed via Definition~\ref{def:Lukasiewicz}.
\end{itemize}
\end{itemize}

We say that $\alpha$ is \emph{$\fourLukProb$-valid} iff $e(\alpha)=1$ in every $\fourLukProb$-model. A set of formulas $\Gamma$ \emph{entails} $\alpha$ ($\Gamma\models_{\fourLukProb}\alpha$) iff there is no $\mathbb{M}$ s.t.\ $e(\gamma)=1$ for every $\gamma\in\Gamma$ but $e(\alpha)\neq1$.
\end{definition}
\begin{definition}[$\LukProbsquare$: language and semantics]\label{def:PrLuk2}
The language of $\LukProbsquare$ is given by the following grammar
\begin{align*}
\LLukProbsquare\ni\alpha&\coloneqq\Prob\phi\mid{\sim}\alpha\mid\neg\alpha\mid\triangle\alpha\mid(\alpha\rightarrow\alpha)\tag{$\phi\in\LBD$}
\end{align*}
A~$\LukProbsquare$-model is a tuple $\mathbb{M}=\langle\mathfrak{M},\mu,e_1,e_2\rangle$ s.t.
\begin{itemize}
\item $\langle\mathfrak{M},\mu\rangle$ is a $\BD$-model with $\pm$-probability;
\item $e_1$ and $e_2$ are \emph{$\Luksquareorder$ valuations induced by $\mu$}, i.e.:
\begin{itemize}
\item $e_1(\Prob\phi)=\mu(|\phi|^+)$, $e_2(\Prob\phi)=\mu(|\phi|^-)$;
\item the values of complex $\LLukProbsquare$ formulas are computed following Definition~\ref{def:Luk2triangle}.
\end{itemize}
\end{itemize}

We say that $\alpha$ is \emph{$\LukProbsquare$-valid} iff $e(\alpha)=(1,0)$ in every $\LukProbsquare$-model. A set of formulas $\Gamma$ \emph{entails} $\alpha$ ($\Gamma\models_{\LukProbsquare}\alpha$) iff there is no $\mathbb{M}$ s.t.\ $e(\gamma)=(1,0)$ for all $\gamma\in\Gamma$ but $e(\alpha)\neq(1,0)$.
\end{definition}

We note quickly that in Definitions~\ref{def:4ProbLuk} and~\ref{def:PrLuk2}, $e$ (as well as $e_1$ and $e_2$) are valuations of \emph{outer-layer} formulas. Thus, they are defined only on \emph{modal atoms}, not propositional variables. Propositional variables, in turn, have their values assigned by $v^+$ and $v^-$ valuations of the $\BD$-model over which the $\fourLukProb$- or $\LukProbsquare$-model is built.

The following property of $\LukProbsquare$ will be useful further in the section.
\begin{lemma}\label{lemma:LukProbsquareconflation}
Let $\alpha\in\LLukProbsquare$. Then, $\alpha$ is $\LukProbsquare$-valid iff $e_1(\alpha)=1$ in every $\LukProbsquare$-model.
\end{lemma}
\begin{proof}
Let $\mathbb{M}=\langle W,v^+,v^-,\mu,e_1,e_2\rangle$ be a $\LukProbsquare$-model s.t.\ $e_2(\alpha)\neq0$. We construct a model $\mathbb{M}^*=\langle W,(v^*)^+,(v^*)^-,\mu,e^*_1,e^*_2\rangle$ where $e^*_1(\alpha)\neq1$. To do this, we define new $\BD$ valuations $(v^*)^+$ and $(v^*)^-$ on $W$ as follows.
\begin{align*}
(v^*)^+&=W\setminus v^-&(v^*)^-&=W\setminus v^+
\end{align*}
% \begin{align*}
% w\in v^+(p),w\notin v^-(p)&\text{ then }w\in(v^*)^+(p),w\notin(v^*)^-(p)\\
% w\in v^+(p),v^-(p)&\text{ then }w\notin(v^*)^+(p),(v^*)^-(p)\\
% w\notin v^+(p),v^-(p)&\text{ then }w\in(v^*)^+(p),(v^*)^-(p)\\
% w\notin v^+(p),w\in v^-(p)&\text{ then }w\notin(v^*)^+(p),w\in(v^*)^-(p)
% \end{align*}
It can be easily checked by induction on $\phi\in\LBD$ that
\begin{align}\label{equ:conflatedmodels}
|\phi|^+_\mathbb{M}&=W\setminus|\phi|^-_{\mathbb{M}^*}&|\phi|^-_\mathbb{M}&=W\setminus|\phi|^+_{\mathbb{M}^*}
\end{align}
Now, since we can w.l.o.g.\ assume that $\mu$ is a~(classical) probability measure on~$W$ (recall Example~\ref{example:+-nomeasure}), we have that
\begin{align}\label{equ:conflatedprobability}
e^*(\Prob\phi)=(1-\mu(|\phi|^-),1-\mu(|\phi|^+))=(1-e_2(\Prob\phi),1-e_1(\Prob\phi))
\end{align}
It remains to show that $e^*(\alpha)=(1-e_2(\alpha),1-e_1(\alpha))$ for every $\alpha\in\LLukProbsquare$. We proceed by induction on formulas. The basis case of $\alpha=\Prob\phi$ holds by~\eqref{equ:conflatedprobability}.

Consider $\alpha=\beta\rightarrow\beta'$.
\begin{align*}
e^*_1(\beta\rightarrow\beta')&=\min(1,1-e^*_1(\beta)+e^*_1(\beta'))\\
&=\min(1,1-(1-e_2(\beta))+(1-e_2(\beta')))\tag{by IH}\\
&=\min(1,1-e_2(\beta')+e_2(\beta))\\
&=1-\max(0,e_2(\beta')-e_2(\beta))\\
&=1-e_2(\beta\rightarrow\beta')
\end{align*}
\begin{align*}
e^*_2(\beta\rightarrow\beta')&=\max(0,e^*_2(\beta')-e^*_2(\beta))\\
&=\max(0,1-e_1(\beta')-(1-e_1(\beta)))\tag{by IH}\\
&=\max(0,e_1(\beta)-e_1(\beta'))\\
&=1-\min(1,1-e_1(\beta)+e_1(\beta'))
\end{align*}
The remaining cases of $\alpha=\neg\beta$, $\alpha={\sim}\beta$, and $\alpha=\triangle\beta$ can be tackled similarly.

It is now clear that if $e(\alpha)=(1,y)$ for some $y>0$, we have $e^*(\alpha)=(1-y,0)$, i.e., $e^*_1(\alpha)<1$, as required. The result follows.
\end{proof}

% The next statement can be obtained using Lemma~\ref{lemma:LukProbsquareconflation}.
% \begin{corollary}\label{cor:LukProbsquareconflationsatisfiability}
% Let $\alpha\in\LLukProbsquare$. Then there is a~pair of valuations $e=\langle e_1,e_2\rangle$ s.t.\ $e(\alpha)=(1,0)$ iff there is a~valuation~$e'_1$ s.t.\ $e'_1(\alpha)=1$.
% \end{corollary}
% \begin{proof}
% The left-to-right direction is straightforward. For the converse, let 
% \end{proof}
\begin{example}[Values of formulas in models]\label{example:formulasvalues}
Consider Fig.~\ref{fig:formulasvalues} with a~classical probability $\mu$. Let us compute the values of the following (atomic) formulas: $\Prob(p\wedge\neg q)$, $\puredisbelmod(p\wedge\neg q)$, and $\purebelmod(q\vee\neg q)$.

We have $|p\wedge\neg q|^+=\{w_1\}$ and $|p\wedge\neg p|^-=\{w_1,w_3\}$. Thus, $\mu(|p\wedge\neg q|^+)=\frac{1}{3}$ and $\mu(|p\wedge\neg q|^-)=\frac{1}{2}$ which gives us $e(\Prob(p\wedge\neg q))=\left(\frac{1}{3},\frac{1}{2}\right)$. Moreover, $|p\wedge\neg q|^\puredisbel=\{w_3\}$, whence $\puredisbelmod(p\wedge\neg q)=\frac{1}{6}$. For $q\vee\neg q$, we have $|q\vee\neg q|^\purebel=\{w_2\}$, and thus, $e(\purebelmod(q\vee\neg q))=\frac{1}{2}$.
\end{example}
\begin{figure}
\centering
\[\xymatrix{w_1:p^+,q^\pm&w_2:\xcancel{p},q^-&w_3:p^-,\xcancel{q}}\]
\vspace{1em}
\caption{$\mu(\{w_1\})=\frac{1}{3}$, $\mu(\{w_2\})=\frac{1}{2}$, $\mu(\{w_3\})=\frac{1}{6}$.}
\label{fig:formulasvalues}
\end{figure}

Before establishing faithful translations between $\LukProbsquare$ and $\fourLukProb$, let us recall that in the \nameref{sec:introduction}, we mentioned that quantitative statements about uncertainty expressed in the natural language may sound like ‘I think that rain is twice more likely than snow’. We show how to formalise this statement.
\begin{example}[Formalisation]\label{example:formalisation}~
\begin{itemize}
\item[]$\mathtt{twice}$: I think that rain is twice more likely than snow.
\end{itemize}
We are going to translate this statement into $\LfourLukProb$ and treat ‘I think that’ as pure belief modality. We denote ‘it is going to rain’ with $r$ and ‘it is going to snow’ with $s$. It remains to write down a~formula $\phi_\mathtt{twice}$ that has value $1$ iff $e(\purebelmod r)=2\cdot e(\purebelmod s)$. Consider the following formula
\begin{align*}
\phi_\mathtt{twice}&=\triangle((\purebelmod r\ominus\purebelmod s)\leftrightarrow\purebelmod s)
\end{align*}
Note that a more intuitive formalisation --- $\triangle(\purebelmod r\leftrightarrow(\purebelmod s\oplus\purebelmod s))$ --- would not work: $\oplus$ is a~\emph{truncated sum}, whence it, e.g., does not exclude the situation with both $\purebelmod r$ and $\purebelmod s$ having value~$1$. It can, however, be altered as follows: $\triangle(\purebelmod r\leftrightarrow(\purebelmod s\oplus\purebelmod s))\wedge{\sim}(\purebelmod s\odot\purebelmod s)$ which will give the desired outcome.
\end{example}

It is also clear that desiderata (1), (3), and (4) listed in Section~\ref{ssec:plan} are satisfied by both $\fourLukProb$ and $\LukProbsquare$. It is less straightforward, however, to see how we can formalise preferring one source to another (desideratum (2)). We explain this in the following remark.
\begin{remark}
To represent different sources, one can consider $\BD$-models with \emph{several measures}, each representing a~source, and, accordingly, expand the language with other modalities corresponding to these new measures. We can state, for example, that one source ($s_1$) considers $\phi$ to be more likely than the other source ($s_2$) does. In $\fourLukProb$ we can interpret this as the value of $\purebelmod_{s_1}\phi$ being \emph{smaller than} the value of $\purebelmod_{s_2}\phi$. This is formalised as follows:
\begin{align*}
{\sim}\triangle(\purebelmod_{s_2}\phi\rightarrow\purebelmod_{s_1}\phi)
\end{align*}
Unfortunately, there seems to be no direct way of representing the \emph{degrees of trust} the agent assigns to $s_1$ and $s_2$ using only modalities treated as measures in the Łukasiewicz setting. A~traditional way (cf., e.g.,~\cite[p.252]{Shafer1976}) of accounting for the degree of trust in a~given source is to multiply the value a~mass function gives to $X\subseteq W$ by some $x\in[0,1]$. Thus, to model this approach, one would need a~combination of Rational Pavelka and Product logics. Another option would be to redefine $e(\purebelmod_s\phi)=(\mathsf{tr}_s\odot_\Luk\mu(|\phi|^\purebel))$ and similarly for other modalities (with $\mathsf{tr}_s\in[0,1]$ standing for the trust in source $s$). It is unclear, however, how to axiomatise this logic.

It is possible, though, to make different modalities stand for different types of measures (e.g., $\purebelmod_{s_1}$ can be generated by a~$\four$-probability while $\purebelmod_{s_2}$ by a~belief function). This represents the \emph{different ways of aggregating the data} the agent can have.
\end{remark}

We finish the section by establishing faithful embeddings of $\fourLukProb$ and $\LukProbsquare$ into one another. First, we introduce some technical notions that will simplify the proof.
\begin{convention}\label{conv:LukProbsquarenegfree}
We say that $\alpha\in\LLukProbsquare$ is \emph{outer-$\neg$-free} when $\neg$'s appear only inside modal atoms.
\end{convention}
\begin{definition}\label{def:positiveNNFs}
Let $\phi\in\LBD$ and $\alpha\in\LLukProbsquare$. $\alpha^\neg$ is produced from $\alpha$ by successively applying the following transformations.
\begin{align*}
\neg\Prob\phi&\rightsquigarrow\Prob\neg\phi&\neg\neg\alpha&\rightsquigarrow\alpha&\neg{\sim}\alpha&\rightsquigarrow{\sim}\neg\alpha\\\neg(\alpha\rightarrow\alpha')&\rightsquigarrow{\sim}(\neg\alpha'\rightarrow\neg\alpha)&\neg\triangle\alpha&\rightsquigarrow{\sim}\triangle{\sim}\neg\alpha
\end{align*}
\end{definition}

It is easy to see that $\alpha^\neg$ is outer-$\neg$-free. Using Definitions~\ref{def:Luk2triangle} and~\ref{def:PrLuk2}, one can also check that $e(\alpha)=e(\alpha^\neg)$ in every $\LukProbsquare$-model.
\begin{definition}\label{def:embeddings}
Let $\phi\in\LBD$ and $\alpha\in\LLukProbsquare$ be outer-$\neg$-free. We define $\alpha^\four\in\LfourLukProb$ as follows.
\begin{align*}
(\Prob\phi)^\four&=\purebelmod\phi\oplus\conflmod\phi\\
(\heartsuit\alpha)^\four&=\heartsuit\alpha^\four\tag{$\heartsuit\in\{\triangle,{\sim}\}$}\\
(\alpha\rightarrow\alpha')^\four&=\alpha^\four\rightarrow\alpha'^\four
\end{align*}

Let $\beta\in\LfourLukProb$. We define $\beta^\pm$ as follows.
\begin{align*}
(\purebelmod\phi)^\pm&=\Prob\phi\ominus\Prob(\phi\wedge\neg\phi)\\
(\conflmod\phi)^\pm&=\Prob(\phi\wedge\neg\phi)\\
(\uncertmod\phi)^\pm&={\sim}\Prob(\phi\vee\neg\phi)\\
(\puredisbelmod\phi)^\pm&=\Prob\neg\phi\ominus\Prob(\phi\wedge\neg\phi)\\
(\heartsuit\beta)^\pm&=\heartsuit\beta^\pm\tag{$\heartsuit\in\{\triangle,{\sim}\}$}\\
(\beta\rightarrow\beta')^\pm&=\beta^\pm\rightarrow\beta'^\pm
\end{align*}
\end{definition}
\begin{theorem}\label{theorem:embeddings1}
$\alpha\in\LLukProbsquare$ is $\LukProbsquare$-valid iff $(\alpha^\neg)^\four$ is $\fourLukProb$-valid.
\end{theorem}
\begin{proof}
Let w.l.o.g.\ $\mathbb{M}=\langle W,v^+,v^-,\mu,e_1,e_2\rangle$ be a $\BD$-model with $\pm$-probability where $\mu$ is a \emph{classical} probability measure and let $e(\alpha)=(x,y)$. We show that in the $\BD$-model $\mathbb{M}_\four=\langle W,v^+,v^-,\mu,e_1\rangle$ with \emph{four-probability} $\mu$ and $e^\four$ induced by $\mu$, $e^\four((\alpha^\neg)^\four)=x$. This is sufficient to prove the result because if $\alpha^\neg$ is not $\LukProbsquare$-valid, then $\alpha$ is not $\LukProbsquare$-valid either.

We proceed by induction on $\alpha^\neg$. For the basis case, we have that
\begin{align*}
e_1(\Prob\phi)&=\mu(|\phi|^+)\\
&=\mu(|\phi|^\purebel\cup|\phi|^\conf)\\
&=\mu(|\phi|^\purebel)\cup\mu(|\phi|^\conf)\tag{$|\phi|^\purebel$ and $|\phi|^\conf$ are disjoint}\\
&=e^\four(\purebelmod\phi)+e_\four(\conflmod\phi)\tag{$e^\four$ is induced by $\mu$}\\
&=e^\four(\purebelmod\phi\oplus\conflmod\phi)\tag{$e^\four(\purebelmod\phi)+e_\four(\conflmod\phi)\leq1$}
\end{align*}
The cases of connectives can be obtained by an application of the induction hypothesis.

Conversely, since the support of truth conditions in $\Luksquareorder$ coincide with the semantics of $\Luktriangle$ (cf.~Definitions~\ref{def:Lukasiewicz} and~\ref{def:Luk2triangle}) and since $\alpha^\neg$ is outer-$\neg$-free, if $e(\alpha^\neg)<1$ for some $\fourLukProb$-model $\mathbb{M}_\four=\langle W,v^+,v^-,\mu,e\rangle$, then $e_1(\alpha^\neg)<1$ for $\mathbb{M}=\langle W,v^+,v^-,\mu,e_1,e_2\rangle$ with $e=e_1$.

We proceed by induction on $\alpha^\neg$ (recall that $e(\alpha)=e(\alpha^\neg)$ in all $\LukProbsquare$-models). If $\alpha=\Prob\phi$, then $e_1(\Prob\phi)=\mu(|\phi|^+)=\mu(|\phi|^\purebel\cup|\phi|^\conf)$. But $|\phi|^\purebel$ and $|\phi|^\conf$ are disjoint, whence $\mu(|\phi|^\purebel\cup|\phi|^\conf)=\mu(|\phi|^\purebel)+\mu(|\phi|^\conf)$, and since $\mu(|\phi|^\purebel)+\mu(|\phi|^\conf)\leq1$, we have that $e_1(\purebelmod\phi\oplus\conflmod\phi)=\mu(|\phi|^\purebel)+\mu(|\phi|^\conf)=e_1(\Prob\phi)$, as required.

The induction steps are straightforward since the semantic conditions of support of truth in $\Luk^2_\triangle$ coincide with the semantics of $\Luk_\triangle$ (cf.~Definitions~\ref{def:Luk2triangle} and~\ref{def:Lukasiewicz}).
\end{proof}
\begin{theorem}\label{theorem:embeddings2}
$\beta\in\LfourLukProb$ is $\fourLukProb$-valid iff $\beta^\pm$ is $\LukProbsquare$-valid.
\end{theorem}
\begin{proof}
Assume w.l.o.g.\ that $\mathbb{M}=\langle W,v^+,v^-,\mu_\four,e\rangle$ is a $\BD$-model with a~$\four$-probability where $\mu_\four$ is a classical probability measure and $e(\beta)=x$. It suffices to define a $\BD$-model with $\pm$-probability $\mathbb{M}^\pm=\langle W,v^+,v^-,\mu_\four,e_1,e_2\rangle$ and show that $e_1(\beta^\pm)=x$. If $e(\beta)<1$, then $e(\beta^\pm)<1$ (and thus, is not valid). If $\beta^\pm$ is not valid, we have that $e_1(\beta^\pm)<1$ by Lemma~\ref{lemma:LukProbsquareconflation}, whence $e(\beta)<1$, as well.

We proceed by induction on $\beta$. If $\beta=\purebelmod\phi$, then $e(\purebelmod\phi)=\mu_\four(|\phi|^\purebel)$. Now observe that $\mu_\four(|\phi|^+)=\mu(|\phi|^\purebel\cup|\phi|^\conf)=\mu_\four(|\phi|^\purebel)+\mu_\four(|\phi|^\conf)$ since $|\phi|^\purebel$ and $|\phi|^\conf$ are disjoint. But $\mu_\four(|\phi|^+)\!=\!e_1(\Prob\phi)$ and $\mu_\four(|\phi|^\conf)\!=\!\mu_\four(|\phi\!\wedge\!\neg\phi|^+)$ since $|\phi\!\wedge\!\neg\phi|^+\!=\!|\phi|^\conf$. Thus, $\mu_\four(|\phi|^\purebel)=e_1(\Prob\phi\ominus\Prob(\phi\wedge\neg\phi))$ as required.

Other basis cases of $\conflmod\phi$, $\uncertmod\phi$, and $\puredisbelmod\phi$ can be tackled similarly. The induction steps are straightforward since the support of truth in $\Luk^2_\triangle$ coincides with semantical conditions in~$\Luk_\triangle$.
\end{proof}

We finish the section with a straightforward observation.
\begin{lemma}\label{lemma:lineartranslationprobability}
Let $\alpha\in\LLukProbsquare$ and $\beta\in\LfourLukProb$. Then $\lmc((\alpha^\neg)^\four)=\Omc(\lmc(\alpha))$ and $\lmc(\beta)=\Omc(\lmc(\beta^\pm))$.
\end{lemma}
\begin{proof}
To simplify the presentation of the proof, we will further assume that $\lmc(\mathsf{Y}\phi)=\lmc(\phi)+2$ for every $\mathsf{Y}\in\{\Prob,\purebelmod,\conflmod,\uncertmod,\puredisbelmod\}$. We begin with $\alpha$. From Definition~\ref{def:positiveNNFs}, it is clear that $\lmc(\alpha^\neg)=\Omc(\lmc(\alpha))$. It remains to show by induction on $\alpha^\neg$ that $\lmc((\alpha^\neg)^\four)=\Omc(\lmc(\alpha^\neg))$.

The basis case of $\alpha^\neg=\Prob\phi$ is simple: $\lmc(\Prob\phi)=\lmc(\phi)+2$, whence
\begin{align*}
\lmc((\Prob\phi)^\four)=\lmc(\purebelmod\phi\oplus\conflmod\phi)=2\cdot\lmc(\phi)+7=\Omc(\lmc(\alpha^\neg))
\end{align*}
Here, $7$ is the sum of lengths of modalities, $\oplus$, and outer brackets in $(\Prob\phi)^\four$.

Now let $\alpha^\neg=(\alpha_1\rightarrow\alpha_2)$. Then $\lmc(\alpha^\neg)=\lmc(\alpha_1)+\lmc(\alpha_2)+3$ (we count outer brackets here). Since $(\alpha_1\rightarrow\alpha_2)^\four=\alpha^\four_1\rightarrow\alpha^\four_2$, we have $\lmc(\alpha^\four_1\rightarrow\alpha^\four_2)=\lmc(\alpha^\four_1)+\lmc(\alpha^\four_2)+3$. We apply the induction hypothesis and obtain that $\lmc(\alpha^\four_1\rightarrow\alpha^\four_2)=\Omc(\lmc(\alpha_1))+\Omc(\lmc(\alpha_2))+3=\Omc(\lmc(\alpha_1\rightarrow\alpha_2))$.

The cases of $\alpha=\triangle\alpha'$ and $\alpha={\sim}\alpha'$ can be dealt with in the same manner.

Consider now $\beta\in\LfourLukProb$. We show by induction on $\beta$ that $\lmc(\beta^\pm)=\Omc(\lmc(\beta))$. Assume that $\beta=\purebelmod\phi$. Then $(\purebelmod\phi)^\pm=\Prob\phi\ominus\Prob(\phi\wedge\neg\phi)$. Observe that $\Prob\phi\ominus\Prob(\phi\wedge\neg\phi)$ is a~shorthand for ${\sim}(\Prob\phi\rightarrow\Prob(\phi\wedge\neg\phi))$. Thus, we have
\begin{align*}
\lmc({\sim}(\Prob\phi\rightarrow\Prob(\phi\wedge\neg\phi)))=3\cdot\lmc(\phi)+12=\Omc(\lmc(\purebelmod\phi))
\end{align*}
Note again that $12$ is the sum of the lengths of modalities, $\wedge$, $\neg$, $\rightarrow$, $\sim$, and brackets in ${\sim}(\Prob\phi\rightarrow\Prob(\phi\wedge\neg\phi))$. The cases of other modal atoms are similar. The cases of $\beta=\beta_1\rightarrow\beta_2$, $\beta=\triangle\beta'$, and $\beta={\sim}\beta'$ can be considered in the same way as we did $(\alpha^\neg)^\four$. The result now follows since modalities do not nest.
\end{proof}
\subsection{Axiomatisations}
In this section, we present Hilbert-style calculi $\HLukProbsquare$ and $\HfourLukProb$ that axiomatise $\LukProbsquare$ and $\fourLukProb$. The completeness of $\HLukProbsquare$ w.r.t.\ finite theories ($\Luk$ is not compact, whence there are no finite calculi for expansions of $\Luk$ complete w.r.t.\ countable theories) was established by~\cite{BilkovaFrittellaKozhemiachenkoMajerNazari2023APAL}. Here, we prove the completeness of $\HfourLukProb$.

To construct the calculi we translate the conditions on measures from Definitions~\ref{def:measuredmodel} and~\ref{def:4model} into $\LLukProbsquare$ and $\LfourLukProb$ formulas. We then use these translated conditions as additional axioms that extend Hilbert-style calculi for $\Luktriangle$ and $\Luksquareorder$.

To facilitate the presentation, we recall these Hilbert calculi for $\Luktriangle$. The axiomatisation of $\Luktriangle$ can be obtained by adding $\triangle$-axioms and rules from~\cite{Baaz1996}, \cite[Definition~2.4.5]{Hajek1998}, or~\cite[Chapter~I,2.2.1]{BehounekCintulaHajek2011MFL1} to the Hilbert-style calculus for $\Luk$ from~\cite[\S6.2]{MetcalfeOlivettiGabbay2008}.
\begin{definition}[$\HLuktriangle$ --- the Hilbert-style calculus for $\Luk_\triangle$]\label{def:HLuk}
The calculus contains the following axioms and rules.
\begin{description}
\item[$\mathbf{w}$:] $\phi\rightarrow(\chi\rightarrow\phi)$.
\item[$\mathbf{sf}$:] $(\phi\rightarrow\chi)\rightarrow((\chi\rightarrow\psi)\rightarrow(\phi\rightarrow\psi))$.
\item[$\mathbf{waj}$:] $((\phi\rightarrow\chi)\rightarrow\chi)\rightarrow((\chi\rightarrow\phi)\rightarrow\phi)$.
\item[$\mathbf{co}$:] $({\sim}\chi\rightarrow{\sim}\phi)\rightarrow(\phi\rightarrow\chi)$.
\item[$\mathbf{MP}$:] $\dfrac{\phi\quad\phi\rightarrow\chi}{\chi}$.
\item[$\triangle1$:] $\triangle\phi\vee {\sim}\triangle\phi$.
\item[$\triangle2$:] $\triangle\phi\rightarrow\phi$.
\item[$\triangle3$:] $\triangle\phi\rightarrow\triangle\triangle\phi$.
\item[$\triangle4$:] $\triangle(\phi\vee\chi)\rightarrow\triangle\phi\vee\triangle\chi$.
\item[$\triangle5$:] $\triangle(\phi\rightarrow\chi)\rightarrow\triangle\phi\rightarrow\triangle\chi$.
\item[$\triangle\mathbf{nec}$:] $\dfrac{\phi}{\triangle\phi}$.
\end{description}
\end{definition}
\begin{proposition}[Finite strong completeness of $\HLuktriangle$]\label{prop:HLuktrianglecompleteness}
Let $\Gamma$ be finite. Then $\Gamma\models_{\Luktriangle}\phi$ iff $\Gamma\vdash_{\HLuktriangle}\phi$.
\end{proposition}
The calculus for $\Luksquareorder$ is acquired by adding the axioms for $\neg$ (cf.~\cite{BilkovaFrittellaKozhemiachenkoMajerNazari2023APAL} for details).
\begin{definition}[$\HLuksquareorder$ --- Hilbert-style calculus for $\Luksquareorder$]\label{def:HLuksquareorder}
The calculus expands $\HLuktriangle$ with the following axioms and rules.
\begin{description}
\item[$\neg\neg$:] $\neg\neg\phi\leftrightarrow\phi$.
\item[$\neg{\sim}$:] $\neg{\sim}\phi\leftrightarrow{\sim}\neg\phi$.
\item[${\sim}\neg\!\rightarrow$:] $({\sim}\neg\phi\rightarrow{\sim}\neg\chi) \leftrightarrow {\sim}\neg(\phi\rightarrow\chi)$.
\item[$\neg\triangle$:] $\neg\triangle\phi\leftrightarrow{\sim}\triangle{\sim}\neg\phi$.
\item[$\mathbf{conf}$:] $\dfrac{\phi}{{\sim}\neg\phi}$.
\end{description}
\end{definition}

The completeness of $\HLuksquareorder$ was shown in~\cite[Lemma~4.16]{BilkovaFrittellaKozhemiachenkoMajerNazari2023APAL} (there, the logic is called $\Luk^2$).

The calculi for $\LukProbsquare$ and $\fourLukProb$ are as follows.
\begin{definition}[$\HLukProbsquare$ --- a~Hilbert-style calculus for $\LukProbsquare$]\label{def:HLukProbsquare}
The calculus has the following axioms and rules.
\begin{description}
\item[$\Luksquareorder$:] $\Luksquareorder$-valid formulas and $\HLuksquareorder$ rules instantiated in $\LLukProbsquare$;
\item[$\pm\mathsf{mon}$:] $\Prob\phi\rightarrow\Prob\chi$ if $\phi\models_{\BD}\chi$;
\item[$\pm\mathsf{neg}$:] $\Prob\neg\phi\leftrightarrow\neg\Prob\phi$;
\item[$\pm\mathsf{ex}$:] $\Prob(\phi\vee\chi)\leftrightarrow((\Prob\phi\ominus\Prob(\phi\wedge\chi))\oplus\Prob\chi)$. 
\end{description}
\end{definition}

It is important to note that $\pm\mathsf{neg}$ \emph{is not} connected to the $\pm$-probability of the complement of an event. Indeed, it contains only $\BD$ negations and is, thus, a~translation of the $\mathbf{neg}$ property from Definition~\ref{def:measuredmodel}.
\begin{proposition}[Finite strong completeness of $\HLukProbsquare$~{\cite[Theorem~4.24]{BilkovaFrittellaKozhemiachenkoMajerNazari2023APAL}}]\label{prop:HLukProbsquarecompleteness}
Let $\Xi\subseteq\LLukProbsquare$ be finite. Then $\Xi\models_{\LukProbsquare}\alpha$ iff $\Xi\vdash_{\HLukProbsquare}\alpha$.
\end{proposition}

\begin{definition}[$\HfourLukProb$ --- Hilbert-style calculus for $\fourLukProb$]\label{def:HfourLukProb}
The calculus $\HfourLukProb$ consists of the following axioms and rules.
\begin{description}
\item[$\Luktriangle$:] $\Luktriangle$-valid formulas and sound rules instantiated in $\LfourLukProb$.
\item[$\four\mathsf{equiv}$:] $\mathsf{X}\phi\leftrightarrow\mathsf{X}\chi$ for every $\phi,\chi\in\LBD$ s.t.\ $\phi\dashv\vdash\chi$ is $\BD$-valid and $\mathsf{X}\in\{\purebelmod,\puredisbelmod,\conflmod,\uncertmod\}$.
\item[$\four\mathsf{contr}$:] ${\sim}\purebelmod(\phi\wedge\neg\phi)$; $\conflmod\phi\leftrightarrow\conflmod(\phi\wedge\neg\phi)$.
\item[$\four\mathsf{neg}$:] $\purebelmod\neg\phi\leftrightarrow\puredisbelmod\phi$; $\conflmod\neg\phi\leftrightarrow\conflmod\phi$.
\item[$\four\mathsf{mon}$:] $(\purebelmod\phi\oplus\conflmod\phi)\rightarrow(\purebelmod\chi\oplus\conflmod\chi)$ for every $\phi,\chi\in\LBD$ s.t.\ $\phi\vdash\chi$ is $\BD$-valid.
\item[$\four\mathsf{part1}$:] $\purebelmod\phi\oplus\puredisbelmod\phi\oplus\conflmod\phi\oplus\uncertmod\phi$.
\item[$\four\mathsf{part2}$:] $((\mathsf{X}_1\phi\oplus\mathsf{X}_2\phi\oplus\mathsf{X}_3\phi\oplus\mathsf{X}_4\phi)\ominus\mathsf{X}_4\phi)\leftrightarrow(\mathsf{X}_1\phi\oplus\mathsf{X}_2\phi\oplus\mathsf{X}_3\phi)$ with $\mathsf{X}_i\neq\mathsf{X}_j$, $\mathsf{X}_i\in\{\purebelmod,\puredisbelmod,\conflmod,\uncertmod\}$.
\item[$\four\mathsf{ex}$:] 
$(\purebelmod(\phi\vee\chi)\oplus\conflmod(\phi\vee\chi))\leftrightarrow((\purebelmod\phi\oplus\conflmod\phi)\ominus(\purebelmod(\phi\wedge\chi)\oplus\conflmod(\phi\wedge\chi))\oplus(\purebelmod\chi\oplus\conflmod\chi))$.
\end{description}
\end{definition}
As one can see, the axioms in the calculi above are translations of properties from Definitions~\ref{def:measuredmodel} and~\ref{def:4model}. In $\HfourLukProb$, we split $\four\mathbf{part}$ in two axioms to ensure that the values of $\purebelmod\phi$, $\puredisbelmod\phi$, $\conflmod\phi$, and $\uncertmod\phi$ sum up exactly to $1$. Note, moreover, that $\four\mathsf{mon}$ and $\four\mathsf{ex}$ are translations of $\pm\mathsf{mon}$ and $\pm\mathsf{ex}$ (cf.~Definition~\ref{def:embeddings}). However, since other $\HfourLukProb$-axioms cannot be obtained as translations of $\HLukProbsquare$-axioms (nor, in fact, \emph{any} $\LukProbsquare$-valid formulas), soundness and completeness of $\HfourLukProb$ cannot be established as a~consequence of Theorem~\ref{theorem:embeddings1} and completeness of $\HLukProbsquare$ (Proposition~\ref{prop:HLukProbsquarecompleteness}).

In the remainder of this section, we prove the completeness of $\HfourLukProb$ w.r.t.\ finite theories.
\begin{theorem}\label{theorem:completeness}
Let $\Xi\subseteq\LfourLukProb$ be finite. Then $\Xi\models_{\fourLukProb}\alpha$ iff $\Xi\vdash_{\HfourLukProb}\alpha$.
\end{theorem}
\begin{proof}
Soundness can be established by the routine check of the validity of the axioms. E.g., for $\four\mathsf{equiv}$, observe that if $\phi$ and $\chi$ are equivalent in $\BD$, then $|\phi|^+=|\chi|^+$ and $|\phi|^-=|\chi|^-$, whence $|\phi|^\mathsf{x}=|\chi|^\mathsf{x}$ for every $\mathsf{x}\in\{\purebel,\puredisbel,\conf,\uncert\}$. Thus, $e(\mathsf{X}\phi)=e(\mathsf{X}\chi)$ for each $\mathsf{X}\in\{\purebelmod,\puredisbelmod,\conflmod,\uncertmod\}$, from where $\mathsf{X}\phi\leftrightarrow\mathsf{X}\chi$ is valid. $\four\mathsf{contr}$, $\four\mathsf{neg}$, and $\four\mathsf{mon}$ are straightforward translations of $\mathbf{contr}$, $\mathbf{neg}$, and $\mathbf{BCmon}$. $\four\mathsf{ex}$ is the translation of $\pm\mathsf{ex}$, whence are valid by Theorem~\ref{theorem:embeddings1}.

Last, consider $\four\mathsf{part1}$ and $\four\mathsf{part2}$. We have $\sum\limits_{\mathsf{x}\in\{\purebel,\puredisbel,\conf,\uncert\}}\!\!\!\!|\phi|^\mathsf{x}=1$. Hence, $\sum\limits_{\mathsf{X}\in\{\purebelmod,\puredisbelmod,\conflmod,\uncertmod\}}\!\!\!\!e(\mathsf{X}\phi)=1$. Thus, we have that $\four\mathsf{part1}$ and $\four\mathsf{part2}$ are valid. 

Let us now prove the completeness. We reason by contraposition. Assume that $\Xi\nvdash_{\HfourLukProb}\alpha$. Now, observe that $\HfourLukProb$ proofs are, actually, $\Luktriangle$ proofs with additional probabilistic axioms. 
Let $\Xi^*$ stand for $\Xi$ extended with probabilistic axioms built over all pairwise non-equivalent $\LBD$ formulas constructed from $\Prop[\Xi\cup\{\alpha\}]$. Clearly, $\Xi^*\nvdash_{\HfourLukProb}\alpha$ either. Moreover, $\Xi^*$ is finite as well since $\BD$ is tabular (and whence, there exist only finitely many pairwise non-equivalent $\LBD$ formulas over a~finite set of variables). Now, by the weak completeness of $\Luktriangle$ (Proposition~\ref{prop:HLuktrianglecompleteness}), there exists an $\Luktriangle$ valuation $e$ s.t.\ $e[\Xi^*]=1$ and $e(\alpha)\neq1$.

It remains to construct a~$\fourLukProb$-model $\mathbb{M}$ falsifying $\Xi^*\models_{\fourLukProb}\alpha$ using~$e$. We proceed as follows. First, we set $W=2^{\Lit[\Xi^*\cup\{\alpha\}]}$, and for every $w\in W$ define $w\in v^+(p)$ iff $p\in w$ and $w\in v^-(p)$ iff $\neg p\in w$. We extend the valuations to $\phi\in\LBD$ in the usual manner. Then, for $\mathsf{X}\phi\in\Sf[\Xi^*\cup\{\alpha\}]$, we set $\mu_\four(|\phi|^\mathsf{x})=e(\mathsf{X}\phi)$ according to modality $\mathsf{X}$.

It remains to extend $\mu_\four$ to the whole $2^W$. Observe, however, that any map from $2^W$ to $[0,1]$ that extends $\mu_\four$ is, in fact, a~$\four$-probability. Indeed, all requirements from Definition~\ref{def:4model} concern \emph{only the extensions of formulas}. But the model is finite, $\BD$ is tabular, and $\Xi^*$ contains all the necessary instances of probabilistic axioms and $e[\Xi^*]=1$, whence all constraints on the formulas are satisfied.
\end{proof}
\subsection{Complexity}
Let us now establish the complexity of $\LukProbsquare$ and $\fourLukProb$. Namely, we prove their $\np$-completeness. Since the mutual embeddings given in Definition~\ref{def:embeddings} result in an at most linear increase in the size of the formulas, it suffices to consider only one of them. In this section, we choose to give a~non-deterministic polynomial algorithm for $\LukProbsquare$ since it contains only one modality which will simplify the presentation. Note, furthermore, that $\Luktriangle$ (and thus, $\Luksquareorder$) are $\np$-hard, thus, $\LukProbsquare$ and $\fourLukProb$ are $\np$-hard as well. Thus, we only need to establish the upper bound.

For the proof, we adapt constraint tableaux for $\Luk^2$ defined by~\cite{BilkovaFrittellaKozhemiachenko2021TABLEAUX} and expand them with rules for $\triangle$. We then adapt the $\np$-completeness proof of $\mathsf{FP}(\Luk)$ of~\cite{HajekTulipani2001} to establish our result.
\begin{definition}[Constraint tableaux for $\Luksquareorder$ --- $\TLuksquareorder$]\label{def:L2triangleconstrainttableaux}
Branches contain \emph{labelled formulas} of the form $\phi\leqslant_1i$, $\phi\leqslant_2 i$, $\phi\geqslant_1i$, or $\phi\geqslant_2i$, and \emph{numerical constraints} of the form $i\leq j$ with $i$~and $j$ (labels) being linear polynomials over $[0,1]$. Each branch can be extended by an application of a~rule below.
\begin{align*}
\neg\!\leqslant_1\!\dfrac{\neg\phi\leqslant_1i}{\phi\leqslant_2i}&&
\neg\!\leqslant_2\!\dfrac{\neg\phi\leqslant_2i}{\phi\leqslant_1i}&&
\neg\!\geqslant_1\!\dfrac{\neg\phi\geqslant_1i}{\phi\geqslant_2i}&&
\neg\!\geqslant_2\!\dfrac{\neg\phi\geqslant_2i}{\phi\geqslant_1i}\\[.5em]
{\sim}\!\leqslant_1\!\dfrac{{\sim}\phi\leqslant_1i}{\phi\geqslant_11-i}
&&
{\sim}\!\leqslant_2\!\dfrac{{\sim}\phi\leqslant_2i}{\phi\geqslant_21-i}
&&
{\sim}\!\geqslant_1\!\dfrac{{\sim}\phi\geqslant_1i}{\phi\leqslant_11-i}
&&
{\sim}\!\geqslant_2\!\dfrac{{\sim}\phi\geqslant_2i}{\phi\leqslant_21-i}\\[.5em]
\triangle\!\geqslant_1\!\dfrac{\triangle\phi\geqslant_1i}{i\leq0\left|\begin{matrix}\phi\geqslant_1j\\j\geq1\end{matrix}\right.}
&&
\triangle\!\leqslant_1\!\dfrac{\triangle\phi\leqslant_1i}{i\geq1\left|\begin{matrix}\phi\leqslant_1j\\j<1\end{matrix}\right.}
&&
\triangle\!\leqslant_2\!\dfrac{\triangle\phi\leqslant_2i}{i\geq1\left|\begin{matrix}\phi\leqslant j\\j\leq0\end{matrix}\right.}
&&
\triangle\!\geqslant_2\!\dfrac{\triangle\phi\geqslant_2i}{i\leq0\left|\begin{matrix}\phi\geqslant j\\j>0\end{matrix}\right.}\\[.5em]
\rightarrow\leqslant_1\dfrac{\phi_1\rightarrow\phi_2\leqslant_1i}{i\!\geq\!1\left|\begin{matrix}\phi_1\!\geqslant_1\!1\!-\!i\!+\!j\\\phi_2\leqslant_1j\\j\leq i\end{matrix}\right.}
&&
\rightarrow\leqslant_2\dfrac{\phi_1\!\rightarrow\!\phi_2\!\leqslant_2\!i}{\begin{matrix}\phi_1\geqslant_2j\\\phi_2\!\leqslant_2\!i\!+\!j\end{matrix}}
&&
\rightarrow\geqslant_1\dfrac{\phi_1\!\rightarrow\!\phi_2\geqslant_1 i}{\begin{matrix}\phi_1\!\leqslant_1\!1\!-\!i\!+\!j\\\phi_2\geqslant_1j\end{matrix}}
&&
\rightarrow\geqslant_2\dfrac{\phi_1\rightarrow\phi_2\geqslant_2i}{i\!\leq\!0\left|\begin{matrix}\phi_1\!\leqslant_2\!j\\\phi_2\!\geqslant_2\!i\!+\!j\\j\leq 1-i\end{matrix}\right.}
\end{align*}
Let $i$'s be linear polynomials that label the formulas in the branch and $x$'s be variables ranging over $[0,1]$. We define the translation $\tmc$ from labelled formulas to linear inequalities as follows:
\begin{align*}
\tmc(\phi\leqslant_1i)=x_\phi^L\leq i;&&
\tmc(\phi\geqslant_1i)=x_\phi^L\geq i;&&
\tmc(\phi\leqslant_2i)=x_\phi^R\leq i;&&
\tmc(\phi\geqslant_2i)=x_\phi^R\geq i
\end{align*}
Let $\triangledown\in\{\leqslant_1,\geqslant_1,\leqslant_2,\geqslant_2\}$. A tableau branch
\[\Bmc=\{\phi_1\triangledown i_1,\ldots,\phi_m\triangledown i_m,k_1\leq l_1,\ldots,k_q\leq l_q\}\]
is \emph{closed} if the system of inequalities
\[\Bmc^\tmc=\{\tmc(\phi_1\triangledown i_1),\ldots,\tmc(\phi_m\triangledown i_m),k_1\leq l_1,\ldots,k_q\leq l_q\}\]
does not have solutions. Otherwise, $\Bmc$ is \emph{open}. If no rule application adds new entries to $\Bmc$, it is called \emph{complete}. A tableau is \emph{closed} if all its branches are closed. $\phi$ has a \emph{$\TLuksquareorder$ proof} if the tableau beginning with $\{\phi\leqslant_1c, c<1\}$ is closed.
\end{definition}
\begin{remark}\label{rem:branchingrules}
Let us briefly remark on how to interpret rules and entries in the tableau. Consider, for instance, the rule $\rightarrow\leqslant_2$. It's meaning is: $v_2(\phi_1\rightarrow\phi_2)\leq i$ iff there is $j \in [0,1]$ s.t.\ $v_2(\phi_1)\geq j$ and $v_2(\phi_2)\leq i+j$. Thus, $\psi\leqslant_1i$ means that $v_1(\phi)\leq i$, $\psi\geqslant_2i$ that $v_2(\psi)\geq i$, etc.

Furthermore, our tableaux rules for $\triangle$ are \emph{necessarily branching} in contrast to the rules for standard $\Luk$ connectives proposed by~\cite{Haehnle1992,Haehnle1994}. This is because all connectives of $\Luk$ are continuous. On the other hand, $\triangle$ is not continuous, hence, there cannot be a non-branching rule for it.
\end{remark}
\begin{definition}[Satisfying valuation of a branch]\label{def:satisfyingvaluation}
Let $v_1$ and $v_2$ be $\Luksquareorder$ valuations. $v_1$~\emph{satisfies} a labelled formula $\phi\leqslant_1i$ ($v_2$ satisfies $\phi\geqslant_2i$) iff $v_\mathrm1(\phi)\leq i$ (resp., $v_2(\phi)\geq i$). $v_1$~\emph{satisfies} a~branch~$\Bmc$ iff $v_1$~\emph{satisfies} every labelled formula in $\Bmc$ (and similarly for $v_2$).
A~branch $\Bmc$ is \emph{satisfiable} iff there is a pair of valuations $\langle v_1,v_2\rangle$ that satisfies it. 
\end{definition}
\begin{theorem}[Completeness of tableaux]\label{theorem:tableauxcompleteness} ~
\begin{enumerate}
\item $\phi$ is $\Luktriangle$-valid iff it has a $\TLuksquareorder$ proof.
\item $\phi$ is $\Luksquareorder$-valid iff it has a $\TLuksquareorder$ proof.
\end{enumerate}
\end{theorem}
\begin{proof}
The proof follows~\cite[Theorem~1]{BilkovaFrittellaKozhemiachenko2021TABLEAUX}. For soundness, we show that if $\langle v_1,v_2\rangle$ satisfies the premise of a~rule, then it also satisfies one of its conclusions. We consider the case of $\triangle\!\geqslant_1$ (the rules for other connectives are tackled by~\cite{BilkovaFrittellaKozhemiachenko2021TABLEAUX}; the remaining rules for $\triangle$ can be dealt with analogously).

Let $\triangle\phi\geqslant_1i$ be satisfied by $\langle v_1,v_2\rangle$. Then, $v_1(\triangle\phi)\geqslant i$. We have two cases: (i) $i=0$ and (ii) $i>0$. In the first case, the left conclusion is satisfied. In the second case, $v_1(\triangle\phi)=1$. Hence, $v_1(\phi)=1$, i.e., the right conclusion is satisfied.

For completeness, note from~\cite[Proposition~1]{BilkovaFrittellaKozhemiachenko2021TABLEAUX} that $\phi$ is $\Luksquareorder$-valid iff $v_1(\phi)=1$ for every $v$. Let us now show that every complete open branch can be satisfied. Assume that $\Bmc$ is a complete open branch. We construct the satisfying valuation as follows. Let $\triangledown\!\in\!\{\leqslant_1,\geqslant_1,\leqslant_2,\geqslant_2\}$ and $p_1,\ldots,p_m$ be the propositional variables appearing in the atomic labelled formulas in~$\Bmc$.

Let $\{p_1\triangledown j_1,\ldots,p_m\triangledown j_n\}$ and $\{k_1\leq l_1,\ldots,k_q\leq l_q\}$ be the sets of all atomic labelled formulas and all numerical constraints in $\Bmc$. Notice that one variable might appear in many atomic labelled formulas, hence we might have $m \neq n$. Since $\Bmc$ is complete and open, the following system of linear inequalities over the set of variables $\{x_{p_1}^L,x_{p_1}^R,\ldots,x_{p_m}^L,x_{p_m}^R\}$ must have at least one solution under the constraints listed:
\[\tmc(p_1\triangledown i_1),\ldots,\tmc(p_m\triangledown i_n),k_1\leq l_1,\ldots,k_q\leq l_q\]
Let $c=(c^L_1,c^R_1,\ldots,c^L_m,c^R_m)$ be a solution to the above system of inequalities such that $c^L_j$ ($c^R_j$) is the value of $x_{p_j}^L$ ($x_{p_j}^R$). Define valuations as follows: $v_1(p_j)=c^L_j$ and $v_2(p_j)=c^R_j$.

It remains to show by induction on $\phi$ that all formulas present at $\Bmc$ are satisfied by $v_1$ and $v_2$. The basis case of variables holds by the construction of $v_1$ and $v_2$. We consider the case of $\triangle\phi\geqslant_1i$ (the cases of other variables can be dealt with similarly).

Assume that $\triangle\phi\geqslant_1i\in\Bmc$. Since $\Bmc$ is complete and open, we have that either (i) $i\leq0\in\Bmc$ or (ii) $\{\phi\geqslant_1j,j\geq1\}\subseteq\Bmc$. In the first case, since there is a~solution for $\Bmc^\tmc$, we have that $i=0$, and thus $\triangle\phi\geqslant_1i$ is satisfied by any $v_1$. In the second case, we have that $\{\phi\geqslant_1j,j\geq1\}$ is satisfied by the induction hypothesis, whence $v_1(\phi)=1$, and thus, $v_1(\triangle\phi)=1$, as required.
\end{proof}
\begin{definition}[Mixed-integer problem]\label{def:bMIP}
Let $\mathbf{x}=(x_1,\ldots,x_k)\in\mathbb{R}^k$ and $\mathbf{y}=(y_1,\ldots,y_m)\in\mathbb{Z}^m$ be variables, $A$ and $B$ be integer matrices, $h$ an integer vector, and $f(\mathbf{x},\mathbf{y})$ be a~$k+m$-place linear function.
\begin{enumerate}
\item \emph{A general MIP} is to find $\mathbf{x}$ and $\mathbf{y}$ s.t.\ $f(\mathbf{x},\mathbf{y})=\min\{f(\mathbf{x},\mathbf{y}):A\mathbf{x}+B\mathbf{y}\geq h\}$.
\item \emph{A bounded MIP (bMIP)} is to find all solutions that belong to $[0,1]$.
\end{enumerate}
\end{definition}
\begin{proposition}\label{prop:bMIPNPcomplete}
Bounded MIP is $\np$-complete.
\end{proposition}
\begin{remark}\label{rem:branchingNP}
The proof of $\Luk$'s $\np$-completeness by~\cite{Haehnle1992,Haehnle1994} uses the reduction of a~Łukasiewicz formula $\phi$ to \emph{one} instance of a~bMIP of the size $\Omc(\lmc(\phi))$. This is because the rules are \emph{linear}. In our case (cf.~Remark~\ref{rem:branchingrules}), we cannot avoid branching. This, however, does not affect the complexity of $\Luktriangle$ and $\Luksquareorder$: we can just non-deterministically guess one branch of our tableau and then solve the system of inequalities corresponding to it. Note that in our setting \emph{all variables can be evaluated over $\mathbb{R}\cap[0,1]$}; in addition, lengths of branches are polynomial in $\lmc(\phi)$. Hence, we can solve the system of inequalities obtained from the branch in the time polynomial w.r.t.\ $\lmc(\phi)$ (and if there is no solution, then the branch is closed and we need to guess again). Thus, $\Luktriangle$ and $\Luksquareorder$ are also $\np$-complete.
\end{remark}

Let us now proceed to the proof of $\np$-completeness of $\LukProbsquare$. First, we show that we can \emph{completely remove} $\neg$'s from $\LLukProbsquare$ formulas while preserving their satisfiability.
\begin{lemma}\label{lemma:LukProbsquarenegremoval}
For any outer-$\neg$-free $\alpha\in\LLukProbsquare$, there exists $\alpha^*$ where $\neg$ does not occur at all s.t.\ $\alpha$~is $\LukProbsquare$-valid iff $\alpha^*$ is $\LukProbsquare$-valid.
\end{lemma}
\begin{proof}
We construct $\alpha^*$ as follows. First, we take every modal atom $\Prob\phi$ (recall that $\phi\in\LBD$) and replace $\phi$ with its $\NNF$. Second, we replace every literal $\neg p$ occurring in $\Prob(\NNF(\phi))$ with a~corresponding fresh variable $p'$. I.e., $\neg p$ is replaced with $p'$, $\neg q$ with $q'$, etc. Outer-layer connectives remain the same. Observe, that this increases the number of symbols in $\alpha$ \emph{at most linearly}: indeed, the transformation of $\phi$ into $\NNF(\phi)$ is linear (recall Remark~\ref{rem:NNF}) and then we replace every occurrence of $\neg p$ with a~fresh variable $p^*$. It remains to check that validity is preserved.

Let $\alpha$ be \emph{not $\LukProbsquare$-valid}, i.e., let $e(\alpha)\neq(1,0)$ at some $\LukProbsquare$-model. By Lemma~\ref{lemma:LukProbsquareconflation}, this is equivalent to $e_1(\alpha)\neq1$ in some $\mathbb{M}=\langle W,v^+,v^-,\mu,e_1,e_2\rangle$. Now, let $\mathbb{M}^*=\langle W,{v^*}^+,{v^*}^-,\mu,e^*_1,e^*_2\rangle$ s.t.\ ${v^*}^+(p)=v^+(p)$ and ${v^*}^+(p')=v^-(p)$. It suffices to show that $e_1(\alpha)=e^*_1(\alpha^*)$.

We proceed by induction on $\alpha$. Let $\alpha=\Prob\phi$ for some $\phi\in\LBD$. We have that $e_1(\alpha)=\mu(|\phi|^+)=\mu(|\NNF(\phi)|^+)$. We check that $|\NNF(\phi)|^+_\mathbb{M}=|\NNF(\phi)^*|^+_{\mathbb{M}^*}$ by induction on $\NNF(\phi)$.

If $\NNF(\phi)=p$, then $p^*=p$ and $v^+(p)={v^*}^+(p)$ by construction. If $\NNF(\phi)=\neg p$, then $(\neg p)^*=p'$ and $|\neg p|^+=v^-(p)={v^*}^+(p')$ by construction. If $\NNF(\phi)=\chi\wedge\psi$, then $\NNF(\phi)=(\chi\wedge\psi)^*=\chi^*\wedge\psi^*$. By the induction hypothesis, $|\chi|^+=|\chi^*|^+$ and $|\psi|^+=|\psi^*|^+$. Hence, $|\chi\wedge\psi|^+_\mathbb{M}=|(\chi\wedge\psi)^*|^+_{\mathbb{M}^*}$. The case of $\NNF(\phi)=\chi\vee\psi$ can be considered in the same way.

It follows that $e_1(\Prob\phi)=e^*_1(\Prob(\NNF(\phi)^*))$. The cases of $\Luksquareorder$ connectives can be proven by a~straightforward application of the induction hypothesis.

For the converse, assume that $\alpha^*$ is not $\LukProbsquare$-valid. Hence, $e_1(\alpha^*)\neq1$ in some model $\mathbb{M}=\langle W,v^+,v^-,\mu,e_1,e_2\rangle$ by Lemma~\ref{lemma:LukProbsquareconflation}. We define $\mathbb{M}^\bullet=\langle W,{v^\bullet}^+,{v^\bullet}^-,\mu,e^\bullet_1,e^\bullet_2\rangle$ with ${v^\bullet}^-(p)=v^+(p')$ and ${v^\bullet}^+(p)=v^+(p)$. We can now show that $e_1(\alpha^*)=e^\bullet_1(\alpha)$ (i.e., $e^\bullet_1(\alpha)\neq1$) as in the previous case. The result follows.
\end{proof}
We can now apply this lemma to adapt the proof of the $\np$-completeness of $\mathsf{FP}(\Luk)$.
\begin{theorem}\label{theorem:LukProbnpcompleteness}
Validity of $\LukProbsquare$ and $\fourLukProb$ is $\conp$-complete.
\end{theorem}
\begin{proof}
Recall that $\LukProbsquare$ and $\fourLukProb$ can be linearly embedded into one another (Theorems~\ref{theorem:embeddings1} and~\ref{theorem:embeddings2} and Lemma~\ref{lemma:lineartranslationprobability}). Thus, it remains to provide a~non-deterministic polynomial algorithm that decides whether a formula is falsifiable (i.e., \emph{non-valid}) for one of these logics. We choose $\LukProbsquare$ since it has only one modality.

Let $\alpha\in\LLukProbsquare$ be over modal atoms $\Prob\phi_1$, \ldots, $\Prob\phi_n$. By Lemma~\ref{lemma:LukProbsquarenegremoval}, we can w.l.o.g.\ assume that $\alpha$ does not contain $\neg$. We show how to construct a~falsifying $\LukProbsquare$-valuation for $\alpha$.

First, we replace every modal atom $\Prob\phi_i$ with a~fresh variable $q_{\phi_i}$. Denote the new formula $\alpha^-$. It is clear that $\lmc(\alpha^-)=\Omc(\lmc(\alpha))$. We construct a~$\LukProbsquare$ tableau beginning with $\{\alpha^-\leqslant_1c,c<1\}$. As $\alpha^-$ does not contain $\neg$, every branch gives us an a~system of linear inequalities that has a~solution iff $\alpha^-$ is $\Luksquareorder$-falsifiable: $\alpha^-$ is $\Luksquareorder$-falsifiable iff at least one system of inequalities corresponding to a~tableau branch has a~solution. Clearly, if $\alpha^-$ is not $\Luksquareorder$-falsifiable, it is not $\LukProbsquare$-falsifiable either.

Now let $\mathfrak{Br}=\{\Bmc_1,\ldots,\Bmc_w\}$ be the set of all open branches in the $\TLuksquareorder$ tableau for~$\alpha^-$. We guess a~$\Bmc\in\mathfrak{Br}$ and consider the following system of (in)equalities:
\begin{align}
\tag{$\mathrm{LI}(1)^\Bmc$}\label{equ:MIP1}
z_1\triangledown t_1,\ldots,z_n\triangledown t_n,k_1\leq k'_1,\ldots,k_r\leq k'_r,m_1\geq1,\ldots,m_s\geq1,m'_1\leq0,\ldots,m'_t\leq0
\end{align}
Here, $z_i$'s correspond to the values of $q_{\phi_i}$'s in $\alpha^-$, and $t_i$'s are linear polynomials that label~$q_{\phi_i}$'s. Numerical constraints give us $k$'s, $k'$'s, $m$'s and $m'$'s. Denote the number of inequalities and the number of variables in~\ref{equ:MIP1} with $l_1$ and $l_2$, respectively. It is clear that $l_1=\Omc(\lmc(\alpha^-))$ and $l_2=\Omc(\lmc(\alpha^-))$.

We need to check whether $z_i$'s are coherent as probabilities of $\phi_i$'s. I.e., that there is a~probability measure $\mu$ on $2^{\Prop(\alpha)}$ s.t.\ $\mu(|\phi_i|^+)=z_i$. Recall again that because of~\cite[Theorem~4]{KleinMajerRad2021}, we can assume that the probabilities are \emph{classical}. Moreover, we can assume that there are $n$ propositional variables in $\alpha$ (we can always add new superfluous variables or modal atoms to $\alpha$ to make their numbers equal).

Now, introduce $2^n$ variables $u_v$ indexed by $n$-letter words over $\{0,1\}$. These denote whether the variables of $\phi_i$'s are true under $v^+$ and thus, correspond to subsets of $\Prop(\alpha)$. E.g., if $\Prop(\alpha)=\{p_1,p_2,p_3,p_4\}$, then $u_{1001}$ encodes $\{p_1,p_4\}$. Note that $\neg$ does not occur in $\alpha^-$ and thus we care only about $e_1$ and $v^+$. Furthermore, while $n$ is the number of $\phi_i$'s, we can add extra modal atoms or variables to make it also the number of variables. We let $a_{i,v}=1$ when $\phi_i$ is true under $v^+$ and $a_{i,v}=0$ otherwise. Now add new equalities to $\mathrm{MIP}(1)^{\Bmc}$, namely
\begin{align}
\tag{$\mathrm{LI}(2\exp)^{\Bmc}$}\label{equ:MIP2exp}
\sum_{v}u_v&=1&\sum_{v}(a_{i,v}\cdot u_v)&=z_i
\end{align}
The new problem~--- \ref{equ:MIP1}$\cup$\ref{equ:MIP2exp}~--- has a~solution over $[0,1]$ iff $\alpha$ is $\LukProbsquare$-satisfiable since~\ref{equ:MIP2exp} encodes a~measure on $\Prop(\alpha)$ while the existence of a~solution for~\ref{equ:MIP1} ensures that $\alpha$ is $\Luksquareorder$-satisfiable. Furthermore, although there are $l_2+2^n+n$ variables in \ref{equ:MIP1}$\cup$\ref{equ:MIP2exp}, it has no more than $l_1+n+1$ (in)equalities. Thus by~\cite[Lemma~2.5]{FaginHalpernMegiddo1990}, it has a~solution with at most $l_1+n+1$ non-zero entries. We guess a~list $L$ of at most $l_1+n+1$ words $v$ (its size is $n\cdot(l_1+n+1)$). We can now compute the values of $a_{i,v}$'s for $i\leq n$ and $v\in L$ and obtain a~new system of inequalities:
\begin{align}
\tag{$\mathrm{LI}(2\mathrm{poly})^\Bmc$}\label{equ:MIP2poly}
\sum_{v\in L}u_v&=1&\sum_{v\in L}(a_{i,v}\cdot u_v)&=z_i
\end{align}
It is clear that \ref{equ:MIP1}$\cup$\ref{equ:MIP2poly} is of polynomial size w.r.t.\ $\lmc(\alpha^-)$ and thus, can be solved in polynomial time. Moreover, \ref{equ:MIP1}$\cup$\ref{equ:MIP2poly} has a~solution iff the values of $\Prob\phi_i$'s occurring on $\Bmc$ are coherent as probabilities. If there is no solution for \ref{equ:MIP1}$\cup$\ref{equ:MIP2poly}, we guess another open branch from the tableau and repeat the procedure. If there is no open branch $\Bmc'\in\mathfrak{Br}$ s.t.\ its corresponding system of inequalities $\mathrm{LI}(1)^{\Bmc'}\cup\mathrm{LI}(2\mathrm{poly})^{\Bmc'}$ has a~solution, then $\alpha$ is $\LukProbsquare$-valid as well by Lemma~\ref{lemma:LukProbsquarenegremoval}.
\end{proof}
\section{Belief and plausibility functions over $\BD$\label{sec:BDbelief}}
In this section, we introduce $\BD$-models with belief and plausibility functions. We adapt the definitions from~\cite{Zhou2013} and~\cite{BilkovaFrittellaKozhemiachenkoMajerNazari2023APAL}.
\begin{definition}[Belief function]\label{def:belief}
A~\emph{belief function} on $W\neq\varnothing$ is a map $\bel:2^W\rightarrow[0,1]$ s.t.
\begin{itemize}
\item $\bel$ is monotone w.r.t.\ $\subseteq$: if $X\subseteq Y$, then $\bel(X)\leq\bel(Y)$;
\item for every $X_1,\ldots,X_k\subseteq W$, it holds that 
\begin{align*}
\bel\left(\bigcup\limits_{1\leq i\leq k}X_i\right)&\geq\sum\limits_{\scriptsize{\begin{matrix}J\subseteq\{1,\ldots,k\}\\J\neq\varnothing\end{matrix}}}(-1)^{|J|+1}\cdot\bel\left(\bigcap\limits_{j\in J}X_j\right)
\end{align*}
\item $\bel(\varnothing)=0$ and $\bel(W)=1$.
\end{itemize}
\end{definition}
\begin{definition}[Plausibility function]\label{def:plausibility}
A~\emph{plausibility function} on $W\!\neq\!\varnothing$ is a map $\pl\!:\!2^W\!\rightarrow\![0,1]$~s.t.
\begin{itemize}
\item $\pl$ is monotone w.r.t.\ $\subseteq$;
\item for every $X_1,\ldots,X_k\subseteq W$, it holds that
\begin{align*}
\pl\left(\bigcap\limits_{1\leq i\leq k}X_i\right)\leq\sum\limits_{\scriptsize{\begin{matrix}J\subseteq\{1,\ldots,k\}\\J\neq\varnothing\end{matrix}}}(-1)^{|J|+1}\cdot\pl\left(\bigcup\limits_{j\in J}X_j\right)
\end{align*}
\item $\pl(\varnothing)=0$ and $\pl(W)=1$.
\end{itemize}
\end{definition}

Recall that every plausibility function $\pl$ on $W$ gives rise to a~belief function $\bel_\pl$ and vice versa: given a belief function $\bel$ on $W$, one can construct a plausibility function $\pl_\bel$.
\begin{align}\label{equ:beliefplausibilityCL}
\pl_\bel(X)&=1-\bel(W\setminus X)&\bel_\pl(X)&=1-\pl(W\setminus X)
\end{align}
Moreover, it was shown in~\cite[Lemmas~2.10 and~2.11]{BilkovaFrittellaKozhemiachenkoMajerNazari2023APAL} that a~similar statement holds for $\BD$-models. I.e., belief and plausibility functions can be defined via one another even without set-theoretic complements.
\begin{align}\label{equ:beliefplausibilityBD}
\pl_\bel(|\phi|^+)&=1-\bel(|\phi|^-)&\bel_\pl(|\phi|^+)&=1-\pl(|\phi|^-)
\end{align}

In what follows, we will be using two kinds of $\BD$-models with belief functions introduced by~\cite{BilkovaFrittellaKozhemiachenkoMajerNazari2023APAL}: in the first one, belief and plausibility will be interdefinable via~\eqref{equ:beliefplausibilityBD}; in the second one, we will assume them to be independent.
\begin{definition}[$\BD$ $\bel$-models]\label{def:BDbeliefmodel}
A $\BD$ $\bel$-model is a tuple $\mathfrak{M}_\bel=\langle\mathfrak{M},\bel\rangle$ with $\mathfrak{M}$ being a~$\BD$-model (recall Definition~\ref{def:BDframesemantics}) and $\bel$ a belief function on $W$.
\end{definition}
\begin{definition}[$\BD$ $\bel/\pl$-models]\label{def:BDbeliefplausibilitymodel}
A $\BD$ $\bel/\pl$-model is a tuple $\mathfrak{M}_{\bel/\pl}=\langle\mathfrak{M},\bel,\pl\rangle$ with $\mathfrak{M}$ being a~$\BD$-model (recall Definition~\ref{def:BDframesemantics}), $\bel$ a belief function on $W$, and $\pl$ a plausibility function on~$W$.
\end{definition}

Observe from Definitions~\ref{def:belief} and~\ref{def:plausibility} that the traditional axiomatisation of belief and plausibility functions is infinite. In the classical case, this can be circumvented following~\cite{GodoHajekEesteva2001IJCAI,GodoHajekEsteva2003} if we define belief in $\phi$ as \emph{the probability of $\Box\phi$} (with $\Box$ being the \emph{non-nesting} $\mathbf{S5}$ modality). Then, we can use~\eqref{equ:beliefplausibilityCL} to define the plausibility of $\phi$ as \emph{the probability of $\lozenge\phi$} since $\lozenge\phi\equiv{\sim}\Box{\sim}\phi$ (here, $\equiv$ is the classical equivalence and ${\sim}$ is the classical negation). In the remainder of the section, we show that the same can be done in $\BD$.

Let us first recall the classical result.
\begin{definition}\label{def:CLuncertaintymodel}
A \emph{classical uncertainty model} is a tuple $\mathfrak{M}=\langle W,v,\mu\rangle$ with $W\neq\varnothing$, $\mu:2^W\rightarrow[0,1]$, and $v:\Prop\rightarrow2^W$ extended to the satisfaction relation $\mathfrak{M},w\vDash\phi$ as follows.
\begin{align*}
\mathfrak{M},w\vDash p&\text{ iff }w\in v(p)&\mathfrak{M},w\vDash{\sim}\phi&\text{ iff }\mathfrak{M},w\nvDash\phi&\mathfrak{M},w\vDash\phi\wedge\chi&\text{ iff }\mathfrak{M},w\vDash\phi\text{ and }\mathfrak{M},w\vDash\chi
\end{align*}
If $\mu$ is a belief function, plausibility, probability, etc., we call $\mathfrak{M}$ \emph{belief}, \emph{plausibility}, \emph{probabilistic}, etc.\ model.

The \emph{extension} of a formula $\phi$ is defined as $\|\phi\|=\{w:\mathfrak{M},w\vDash\phi\}$.
\end{definition}
\begin{definition}[Classical probabilistic Kripke model]\label{def:CLprobabilisticKripkemodel}
A \emph{classical probabilistic Kripke model} is a~tuple $\mathfrak{M}=\langle W,R,v,\pi\rangle$ with $W\neq\varnothing$, $R$ being an equivalence relation on $W$, $\pi$ being a classical probability measure on $W$ and $v:\Prop\rightarrow2^W$ extended to the satisfaction relation $\mathfrak{M},w\vDash\phi$ as follows.
\begin{align*}
\mathfrak{M},w\vDash p&\text{ iff }w\in v(p)\\
\mathfrak{M},w\vDash{\sim}\phi&\text{ iff }\mathfrak{M},w\nvDash\phi&\mathfrak{M},w\vDash\phi\wedge\chi&\text{ iff }\mathfrak{M},w\vDash\phi\text{ and }\mathfrak{M},w\vDash\chi\\
\mathfrak{M},w\vDash\Box\phi&\text{ iff }\forall w':wRw'\Rightarrow\mathfrak{M},w'\vDash\phi&\mathfrak{M},w\vDash\lozenge\phi&\text{ iff }\exists w':wRw'~\&~\mathfrak{M},w'\vDash\phi
\end{align*}
\end{definition}

\begin{theorem}[{\cite[Theorem~1]{GodoHajekEesteva2001IJCAI}}]\label{theorem:CLmodalprobability}
Let $\phi$ be a \emph{propositional classical formula} and let $\mathfrak{M}_\bel=\langle W,v,\bel\rangle$ and $\mathfrak{M}=\langle W,v,\pl\rangle$ be, respectively, a~belief and plausibility models, then:
\begin{enumerate}
\item[(i)] there is a classical probabilistic model $\mathfrak{M}_\Box\!=\!\langle W_\Box,\!R_\Box,\!v_\Box,\!\pi_\Box\rangle$ s.t.\ $\pi_\Box(\|\Box\phi\|)\!=\!\bel(\|\phi\|)$;
\item[(ii)] there is a classical probabilistic model $\mathfrak{M}_\lozenge=\langle W_\lozenge,R_\lozenge,v_\lozenge,\pi_\lozenge\rangle$ s.t.\ $\pi_\lozenge(\|\lozenge\phi\|)=\pl(\|\phi\|)$.
\end{enumerate}
\end{theorem}

We draw the attention of our readers to the fact that $\Box$ and $\lozenge$ do not have to be $\mathbf{S5}$: it suffices that the underlying modal logic be normal. The original motivation for choosing $\mathbf{S5}$ was that it is locally finite. Note, however, that local finiteness can be achieved in other modal logics as well because $\Box$ and $\lozenge$ do not nest (cf.~\cite[\S6]{FlaminioGodoMarchioni2013} for a~detailed discussion).

Moreover, originally only (i) is proven by~\cite{GodoHajekEesteva2001IJCAI,GodoHajekEsteva2003}. Observe, however, that (ii) can be obtained from~(i) using~\eqref{equ:beliefplausibilityCL} since
\begin{align*}
\pi(\|\lozenge\phi\|)&=\pi(\|{\sim}\Box{\sim}\phi\|)\\
&=1-\pi(\|\Box{\sim}\phi\|)\\
&=1-\bel(\|{\sim}\phi\|)\\
&=1-\bel(W\setminus\|\phi\|)\\
&=\pl_\bel(\|\phi\|)
\end{align*}

Let us now introduce probabilistic models for modal $\BD$ formulas. We borrow the definition of modalities in $\BD$ from~\cite{OdintsovWansing2017}.
\begin{definition}[$\BD$ probabilistic Kripke models]\label{def:BDprobabilisticKripkemodel}
A \emph{$\BD$ probabilistic Kripke model} is a~tuple $\mathfrak{M}=\langle W,R,v^+,v^-,\pi\rangle$ with $W\neq\varnothing$, $R$ being an equivalence relation on $W$, $\pi$ being \emph{a classical probability measure} on $W$ and $v^+,v^-:\Prop\rightarrow2^W$ extended to $\mathfrak{M},w\vDash^+\phi$ and $\mathfrak{M},w\vDash^-\phi$ as in Definition~\ref{def:BDframesemantics} for propositional connectives and as below for modalities.
\begin{align*}
\mathfrak{M},w\vDash^+\Box\phi&\text{ iff }\forall w':wRw'\Rightarrow \mathfrak{M},w'\vDash^+\phi&\mathfrak{M},w\vDash^+\lozenge\phi&\text{ iff }\exists w':wRw'~\&~\mathfrak{M},w'\vDash^+\phi\\
\mathfrak{M},w\vDash^-\Box\phi&\text{ iff }\exists w':wRw'~\&~\mathfrak{M},w'\vDash^-\phi&\mathfrak{M},w\vDash^-\lozenge\phi&\text{ iff }\forall w':wRw'\Rightarrow\mathfrak{M},w'\vDash^-\phi
\end{align*}
Positive and negative extensions of $\phi$ are as in Definition~\ref{def:BDframesemantics}.

We also call the $\langle W,R,v^+,v^-\rangle$ reduct a~\emph{$\BD$ Kripke model}.
\end{definition}
\begin{remark}\label{rem:BDmodalprobabilitytobeliefplausibility}
Since $\vDash^+$ conditions for $\Box$ and $\lozenge$ in the classical logic and in $\BD$ coincide, it is clear (cf.~\cite{Hajek1996} for a~detailed discussion of the classical case) that given a probability $\pi$ on a~Kripke model $\mathfrak{M}$, there are a~belief and plausibility functions $\bel$ and $\pl$ on $\mathfrak{M}$ s.t.\ $\bel(|\phi|^+)=\pi(|\Box\phi|^+)$ and $\pl(|\phi|^+)=\pi(|\lozenge\phi|^+)$ for every $\phi\in\LBD$.
\end{remark}
\begin{theorem}\label{theorem:BDmodalprobability}
Let $\phi_1,\ldots,\phi_n\in\LBD$ and consider a~$\BD$ $\bel$-model $\mathfrak{M}_\bel=\langle\mathfrak{M},\bel\rangle$ and a $\BD$ $\bel/\pl$-model $\mathfrak{M}'_{\bel/\pl}=\langle\mathfrak{M}',\bel',\pl'\rangle$. Then, there exist
\begin{enumerate}
\item a $\BD$ probabilistic Kripke model $\mathfrak{M}_\Box=\langle W_\Box,R_\Box,v^+_\Box,v^-_\Box,\pi_\Box\rangle$ with $\bel(|\phi_i|^+)=\pi_\Box(|\Box\phi_i|^+)$ for each $i\in\{1,\ldots,n\}$;
\item a $\BD$ probabilistic Kripke model $\mathfrak{M}_{\Box,\lozenge}=\langle W_{\Box,\lozenge},R_1,R_2,v^+_{\Box,\lozenge},v^-_{\Box,\lozenge},\pi_{\Box,\lozenge}\rangle$ with $\bel'(|\phi_i|^+)=\pi_{\Box,\lozenge}(|\Box_1\phi_i|^+)$ and $\pl'(|\phi_i|^+)=\pi_{\Box,\lozenge}(|\lozenge_2\phi_i|^+)$ for each $i\in\{1,\ldots,n\}$ (here, $\Box_1$ and $\lozenge_2$ are associated to $R_1$ and $R_2$, respectively).
\end{enumerate}
\end{theorem}
\begin{proof}
We prove (1) as (2) can be dealt with similarly. Let w.l.o.g.\ $\phi$ be in $\NNF$ and denote $\phi^*$ the result of replacing every negated variable $\neg p$ occurring in $\phi$ with a fresh variable~$p^*$. Consider $\mathfrak{M}_\bel=\langle W,v^+,v^-,\bel\rangle$ and $\mathfrak{M}^*_\bel$ (cf.~the proof of Lemma~\ref{lemma:LukProbsquarenegremoval}). It is easy to see that $\bel(|\phi_i|^+)=\bel(|\phi^*_i|^+)$ for every~$i$.

Now define a \emph{classical} belief model $\mathfrak{M}^\cl_\bel=\langle W,v^\cl,\bel\rangle$ with $v^\cl(q)=v^+(q)$ for every $q\in\Prop$. Clearly, $\bel(|\phi^*_i|^+)=\bel(\|\phi^*_i\|)$ for all $\phi^*_i$'s. Thus, by Theorem~\ref{theorem:CLmodalprobability}, we have that there is a~classical Kripke probabilistic model $\mathfrak{M}_\Box\!=\!\langle W_\Box,\!R_\Box,\!v_\Box,\!\pi_\Box\rangle$ s.t.\ $\pi_\Box(\|\Box\phi^*_i\|)\!=\!\bel(\|\phi^*_i\|)$ for every~$i$. It remains to construct a suitable $\BD$ probabilistic Kripke model.

We define $\mathfrak{M}^\BD_\Box=\langle W_\Box,R_\Box,v^+_\Box,v^-_\Box,\pi_\Box\rangle$ with $v^+_\Box(p)=v_\Box(p)$ and $v^-_\Box(p)=v_\Box(p^*)$. One can show by induction on $\phi^*_i$'s that $\|\phi^*_i\|=|\phi_i|^+$, and thus $\|\Box\phi^*_i\|=|\Box\phi_i|^+$ for every $i$: indeed, cf.~semantical conditions for $\Box$ in Definitions~\ref{def:CLprobabilisticKripkemodel} and~\ref{def:BDprobabilisticKripkemodel}. Hence, $\pi_\Box(\|\Box\phi^*_i\|)=\pi_\Box(|\Box\phi_i|^+)$, as required.
\end{proof}
We finish the section with a~brief observation.
\begin{remark}\label{rem:whyS5}
Just as in the classical case, $\Box$ and~$\lozenge$ do not need to be $\mathbf{S5}$. We choose $\BD$ version of $\mathbf{S5}$ because of the following property: if $\mathfrak{M}=\langle W,R,v^+,v^-,\pi\rangle$ is a~$\BD$ probabilistic Kripke model and $\psi$ is a~modal formula where all propositional variables are in the scope of a~modality, and $wRw'$, then $\mathfrak{M},w\vDash^+\psi$ iff $\mathfrak{M},w'\vDash^+\psi$ and $\mathfrak{M},w\vDash^-\psi$ iff $\mathfrak{M},w'\vDash^-\psi$.

% A~similar property holds in bi-relational models. Namely, given $wR_1w'$ and $wR_2w'$, we have $\mathfrak{M},w\vDash^+\psi$ iff $\mathfrak{M},w'\vDash^+\psi$ and $\mathfrak{M},w\vDash^-\psi$ iff $\mathfrak{M},w'\vDash^-\psi$.
\end{remark}
\section{Logics for belief and plausibility functions over $\BD$\label{sec:twolayeredbelief}}
In this section, we recall two-layered logics $\BelLuksquareorder$ and $\BelLuksquareNelson$ for reasoning with belief and plausibility functions over $\BD$ that were presented by~\cite{BilkovaFrittellaKozhemiachenkoMajerNazari2023APAL}. We then combine the technique of~\cite{HajekTulipani2001} with the results of~\cite{GodoHajekEesteva2001IJCAI,GodoHajekEsteva2003} and Theorem~\ref{theorem:BDmodalprobability} to obtain the $\conp$-completeness of $\BelLuksquareorder$ and $\BelLuksquareNelson$.
\subsection{Languages and semantics}
\begin{definition}[$\BelLuksquareorder$: language and semantics]\label{def:BelLuksquareorder} The language of $\BelLuksquareorder$ is constructed using the grammar below.
\begin{align*}
\LBelLuksquareorder\ni\alpha&\coloneqq\Bel\phi\mid\neg\alpha\mid\alpha\rightarrow\alpha\mid{\sim}\alpha\mid\triangle\alpha\tag{$\phi\in\LBD$}
\end{align*}

A \emph{$\BelLuksquareorder$-model} is a tuple $\mathbb{M}_\bel=\langle\mathfrak{M},\bel,e_1,e_2\rangle$ with
\begin{itemize}
\item $\langle\mathfrak{M},\bel\rangle$ being a~$\BD$ $\bel$-model (cf.~Definition~\ref{def:BDbeliefmodel});
\item $e_1$ and $e_2$ being $\Luksquareorder$ valuations induced by $\bel$:
\begin{itemize}
\item $e_1(\Bel\phi)=\bel(|\phi|^+)$, $e_2(\Bel\phi)=\bel(|\phi|^-)$;
\item values of complex $\LBelLuksquareorder$-formulas are computed via Definition~\ref{def:Luk2triangle}.
\end{itemize}
\end{itemize}

We say that $\alpha$ is \emph{$\BelLuksquareorder$-valid} iff $e_1(\alpha)=1$ and $e_2(\alpha)=0$ in all models. A set of formulas $\Gamma$ \emph{entails}~$\alpha$ ($\Gamma\models_{\BelLuksquareorder}\alpha$) iff there is no $\BelLuksquareorder$-model s.t.\ $e(\gamma)=(1,0)$ for all $\gamma\in\Gamma$ but $e(\alpha)\neq(1,0)$.
\end{definition}
\begin{remark}\label{rem:noconflation}
Note that we cannot utilise the technique from Lemma~\ref{lemma:LukProbsquareconflation} to reduce $\BelLuksquareorder$-validity to checking whether $e_1(\alpha)=1$ in every model. This is because belief functions are \emph{not additive}. Thus, even though we can construct $(v^*)^+$ and $(v^*)^-$ s.t.~\eqref{equ:conflatedmodels} holds, we cannot use it to infer the following counterpart of~\eqref{equ:conflatedprobability}:
\begin{align}\label{equ:conflatedbelief}
e^*(\Bel\phi)=(1-\bel(|\phi|^-),1-\bel(|\phi|^+))=(1-e_2(\Bel\phi),1-e_1(\Prob\phi))
\end{align}
Indeed, \eqref{equ:conflatedbelief} does not hold in general since it is not necessarily the case that $\bel(W\setminus X)=1-\bel(X)$ for $X\subseteq W$.
\end{remark}

$\BelLuksquareNelson$ uses a~different paraconsistent expansion of Łukasiewicz logic ($\NLuk$ ) that is inspired by Nelson's paraconsistent logic from~\cite{Nelson1949}. The logic was introduced by~\cite{BilkovaFrittellaKozhemiachenko2021TABLEAUX} and further investigated by~\cite{BilkovaFrittellaKozhemiachenkoMajerNazari2023APAL}. We recall its language and semantics below.
\begin{definition}[$\NLuk$: language and semantics]\label{def:NLuk}
The language is constructed via the following grammar.
\[\LNLuk\ni\phi\coloneqq p\mid\Ninvol\phi\mid\neg\phi\mid(\phi\wedge\phi)\mid(\phi\weakrightarrow\phi)\]
The support of truth and support of falsity conditions are given by the following extensions of $v_1,v_2:\Prop\rightarrow[0,1]$ ($\NLuk$ valuations\index{valuation!$\Luk^2$-valuation!$\NLuk$-valuation}) to the complex formulas.
\begin{align*}
v_1(\neg\phi)&=v_2(\phi)&v_2(\neg\phi)&=v_1(\phi)\\
v_1(\Ninvol\phi)&={\sim_\Luk}v_1(\phi)&v_2(\Ninvol\phi)&=v_1(\phi)\\
v_1(\phi\wedge\chi)&=v_1(\phi)\wedge_\Luk v_1(\chi)&v_2(\phi\wedge\chi)&=v_2(\phi)\vee_\Luk v_2(\chi)\\
v_1(\phi\weakrightarrow\chi)&=v_1(\phi)\rightarrow_\Luk v_1(\chi)&v_2(\phi\weakrightarrow\chi)&=v_1(\phi)\odot_\Luk v_2(\chi)
\end{align*}

We say that $\phi$ is \emph{$\NLuk$-valid} iff $v_1(\phi)=1$ for every $v_1$. $\Gamma$ \emph{entails} $\chi$ ($\Gamma\models_{\NLuk}\chi$) iff there is no $v_1$ s.t.\ $v_1(\phi)=1$ for every $\phi\in\Gamma$ and $v_1(\chi)\neq1$.
\end{definition}
\begin{definition}[$\BelLuksquareNelson$: language and semantics]\label{def:BelLuksquareNelson}
The language is given by the grammar below.
\begin{align*}
\LBelLuksquareNelson\ni\alpha&\coloneqq \Bel\phi\mid\Pl\phi\mid\neg\alpha\mid\alpha\wedge\alpha\mid\alpha\weakrightarrow\alpha\mid\Ninvol\alpha\tag{$\phi\in\LBD$}
\end{align*}

A \emph{$\BelLuksquareNelson$-model} is a~tuple $\mathbb{M}_{\bel/\pl}=\langle\mathfrak{M},\bel,\pl,e_1,e_2\rangle$ with
\begin{itemize}
\item $\langle\mathfrak{M},\bel,\pl\rangle$ being a $\BD$ $\bel/\pl$-model (cf.~Definition~\ref{def:BDbeliefplausibilitymodel});
\item $e_1$ and $e_2$ being $\NLuk$ valuations induced by $\bel$ and $\pl$:
\begin{itemize}
\item $e_1(\Bel\phi)=\bel(|\phi|^+)$, $e_2(\Bel\phi)=\pl(|\phi|^-)$, $e_1(\Pl\phi)=\pl(|\phi|^+)$, $e_2(\Pl\phi)=\bel(|\phi|^-)$;
\item values of complex $\LBelLuksquareNelson$-formulas are computed via Definition~\ref{def:NLuk}.
\end{itemize}
\end{itemize}
We say that $\alpha$ is \emph{$\BelLuksquareNelson$-valid} iff $e_1(\alpha)=1$ for every $\BelLuksquareNelson$-model. $\Gamma$ \emph{entails} $\alpha$ ($\Gamma\models_{\BelLuksquareNelson}\alpha$) iff there is no $\BelLuksquareNelson$-model s.t.\ $e_1(\gamma)=1$ for every $\gamma\in\Gamma$ and $e_1(\alpha)\neq1$.
\end{definition}

Note that using~\eqref{equ:beliefplausibilityBD}, we can define a~plausibility operator $\Pl_\Bel$ in $\BelLuksquareorder$ as $\Pl_\Bel\phi\coloneqq{\sim}\Bel\neg\phi$. The main difference between it and $\Pl$ in $\BelLuksquareNelson$ is that the latter is \emph{independent of belief}. We refer readers to~\cite{BilkovaFrittellaKozhemiachenkoMajerNazari2023APAL} for a more detailed discussion of differences between $\BelLuksquareorder$ and $\BelLuksquareNelson$.

Moreover, $\LukProbsquare$ is complete w.r.t.\ models with \emph{classical} probability measures (recall~\cite[Theorems~3--4]{KleinMajerRad2021} and~\cite[Theorem~4.24]{BilkovaFrittellaKozhemiachenkoMajerNazari2023APAL}). Thus, as belief and plausibility functions are generalisations of probabilities, it means that if a~statement about belief functions is valid, then it is valid about probabilities as well. Formally, let $\alpha\in\LBelLuksquareorder$ and let further $\alpha^\Prob$ be the result of the replacement of $\Bel\phi$'s with $\Prob\phi$'s. Then if $\BelLuksquareorder\models\alpha$, it follows that $\LukProbsquare\models\alpha^\Prob$. This means that $\LukProbsquare$ can be thought of as an extension or a~semantical restriction of $\BelLuksquareorder$.

In the previous section, we showed (Theorem~\ref{theorem:BDmodalprobability}) that belief and plausibility in $\BD$ can be represented as probabilities of modal formulas. We are going to introduce two logics --- $\ProbLukordermodal$ and $\ProbLukNelsonmodal$ --- that deal with those and show that $\BelLuksquareorder$ and $\BelLuksquareNelson$ can be embedded into them.

\begin{definition}[$\ProbLukordermodal$: language and semantics]\label{def:ProbLukordermodal}
The language is constructed as follows.
\begin{align*}
\LProbLukordermodal\ni\alpha&\coloneqq\Prob\heartsuit\phi\mid{\sim}\alpha\mid\neg\alpha\mid\triangle\alpha\mid(\alpha\rightarrow\alpha)\tag{$\phi\in\LBD, \heartsuit\in\{\Box,\lozenge\}$}
\end{align*}

A \emph{$\ProbLukordermodal$-model} is a tuple $\mathbb{M}_\pi=\langle\mathfrak{M},\pi,e_1,e_2\rangle$ s.t.
\begin{itemize}
\item $\langle\mathfrak{M},\pi\rangle$ is a~$\BD$ probabilistic Kripke model;
\item $e_1$ and $e_2$ are $\Luksquareorder$ valuations induced by $\pi$:
\begin{itemize}
\item $e_1(\Prob\heartsuit\phi)=\pi(|\heartsuit\phi|^+)$, $e_2(\Prob\heartsuit\phi)=\pi(|\heartsuit\phi|^-)$ with $\heartsuit\in\{\Box,\lozenge\}$;
\item values of complex $\LProbLukordermodal$-formulas are computed via Definition~\ref{def:Luk2triangle}.
\end{itemize}
\end{itemize}

We say that $\alpha$ is \emph{$\ProbLukordermodal$-valid} iff $e_1(\alpha)=1$ and $e_2(\alpha)=0$ in all $\ProbLukordermodal$-models. A set of formulas $\Gamma$ \emph{entails}~$\alpha$ ($\Gamma\models_{\ProbLukordermodal}\alpha$) iff there is no $\ProbLukordermodal$-model s.t.\ $e(\gamma)=(1,0)$ for all $\gamma\in\Gamma$ but $e(\alpha)\neq(1,0)$.
\end{definition}

The next statements can be proven in the same manner as Lemmas~\ref{lemma:LukProbsquareconflation} and~\ref{lemma:LukProbsquarenegremoval}, respectively.
\begin{lemma}\label{lemma:ProbLukordermodalconflation}
Let $\alpha\in\LProbLukordermodal$. Then $\alpha$ is $\ProbLukordermodal$-valid iff $e_1(\alpha)=1$ in all $\ProbLukordermodal$-models.
\end{lemma}
\begin{lemma}\label{lemma:ProbLukordermodalnegremoval}
Let $\alpha\in\LProbLukordermodal$. Then there is $\alpha^*$ where $\neg$ does not occur at all s.t.\ $\alpha^*$ is $\ProbLukordermodal$-valid iff $\alpha$ is $\ProbLukordermodal$-valid.
\end{lemma}
\begin{definition}[$\ProbLukNelsonmodal$: language and semantics]\label{def:ProbLukNelsonmodal}
The language is generated by the following grammar.
\begin{align*}
\LProbLukNelsonmodal\ni\alpha&\coloneqq\Prob_i\heartsuit_i\phi\mid\neg\alpha\mid\alpha\wedge\alpha\mid\alpha\weakrightarrow\alpha\mid\Ninvol\alpha\tag{$\phi\in\LBD,\heartsuit\in\{\Box,\lozenge\},i\in\{1,2\}$}
\end{align*}

A \emph{$\ProbLukNelsonmodal$-model is a tuple} $\mathbb{M}^{\NLuk}_{\pi_1,\pi_2}=\langle W,R_1,R_2,v^+,v^-,\pi_1,\pi_2,e_1,e_2\rangle$ s.t.
\begin{itemize}
\item $\langle W,R_1,R_2,v^+,v^-,\pi_1,\pi_2\rangle$ is a $\BD$ probabilistic Kripke model with two relations and two measures;
\item $e_1$ and $e_2$ are $\NLuk$ valuations induced by $\pi_1$ and $\pi_2$:
\begin{itemize}
\item $e_1(\Prob_1\heartsuit_1\phi)=\pi_1(|\heartsuit_1\phi|^+)$, $e_2(\Prob_1\heartsuit_1\phi)=\pi_1(|\heartsuit_1\phi|^-)$ with $\heartsuit\in\{\Box,\lozenge\}$;
\item $e_1(\Prob_2\heartsuit_2\phi)=\pi_2(|\heartsuit_2\phi|^+)$, $e_2(\Prob_2\heartsuit_2\phi)=\pi_2(|\heartsuit_2\phi|^-)$ with $\heartsuit\in\{\lozenge,\Box\}$;
\item values of complex $\LProbLukNelsonmodal$-formulas are computed via Definition~\ref{def:NLuk}.
\end{itemize}
\end{itemize}

We say that $\alpha$ is \emph{$\ProbLukNelsonmodal$-valid} iff $e_1(\alpha)=1$ in all $\ProbLukNelsonmodal$-models. A set of formulas $\Gamma$ \emph{entails}~$\alpha$ ($\Gamma\models_{\ProbLukordermodal}\alpha$) iff there is no $\ProbLukNelsonmodal$-model s.t.\ $e(\gamma)=(1,0)$ for all $\gamma\in\Gamma$ but $e(\alpha)\neq(1,0)$.
\end{definition}
\subsection{Embedding of belief logics into probabilistic logics}
Theorem~\ref{theorem:BDmodalprobability} shows that beliefs can be represented as probabilities of modal formulas. In this section, we use this result to prove that we can faithfully embed $\BelLuksquareorder$ and $\BelLuksquareNelson$ into $\ProbLukordermodal$ and~$\ProbLukNelsonmodal$. To simplify the presentation, observe that all $\LBelLuksquareorder$- and $\LBelLuksquareNelson$-formulas can be transformed into an outer-$\neg$-free form since $\Luksquareorder$ and $\NLuk$ permit $\NNF$'s and since the following holds.
\begin{align*}
e(\neg\Bel\phi)&=e(\Bel\neg\phi)\tag{in $\BelLuksquareorder$}\\
e(\neg\Bel\phi)&=e(\Pl\neg\phi)&e(\neg\Pl\phi)&=e(\Bel\neg\phi)\tag{in $\BelLuksquareNelson$}
\end{align*}
\begin{definition}[Embedding of $\BelLuksquareorder$ into $\ProbLukordermodal$]\label{def:BelLuksquareordertoProbLukordermodal}
Let $\alpha\in\LBelLuksquareorder$ be outer-$\neg$-free, we define $\alpha^\boxplus$ and $\alpha^\boxminus$ as follows.
\begin{align*}
(\Bel\phi)^\boxplus&=\Prob(\Box\phi)&(\Bel\phi)^\boxminus&=\Prob(\Box\neg\phi)\\
(\triangle\beta)^\boxplus&=\triangle(\beta^\boxplus)&(\triangle\beta)^\boxminus&={\sim}\triangle{\sim}(\beta^\boxminus)\\
({\sim}\beta)^\boxplus&={\sim}(\beta^\boxplus)&({\sim}\beta)^\boxminus&={\sim}(\beta^\boxminus)\\
(\beta\rightarrow\gamma)^\boxplus&=\beta^\boxplus\rightarrow\gamma^\boxplus&(\beta\rightarrow\gamma)^\boxminus&=\gamma^\boxminus\ominus\beta^\boxminus
\end{align*}
\end{definition}
\begin{definition}[Embedding of $\BelLuksquareNelson$ into $\ProbLukNelsonmodal$]\label{def:BelLuksquareNelsontoProbLukNelsonmodal}
Let $\alpha\in\LBelLuksquareNelson$ be outer-$\neg$-free, we define $\alpha^{\Box,\lozenge}$ as follows.
\begin{align*}
(\Bel\phi)^{\Box,\lozenge}&=\Prob_1(\Box_1\phi)&(\Pl\phi)^{\Box,\lozenge}&=\Prob_2(\lozenge_2\phi)\\
(\Ninvol\beta)^{\Box,\lozenge}&=\Ninvol(\beta^{\Box,\lozenge})&(\beta\circ\gamma)^{\Box,\lozenge}&=\beta^{\Box,\lozenge}\circ\gamma^{\Box,\lozenge}\tag{$\circ\in\{\wedge,\weakrightarrow\}$}
\end{align*}
\end{definition}

We can now prove the following statement.
\begin{theorem}\label{theorem:beliefprobabilitytwolayeredembedding}~
\begin{enumerate}
\item $\alpha\in\LBelLuksquareorder$ is $\BelLuksquareorder$-valid iff (a) $e_1(\alpha^\boxplus)=1$ in every $\ProbLukordermodal$-model and (b) $e_1(\alpha^\boxminus)=0$ in every $\ProbLukordermodal$-model.
\item $\alpha\in\LBelLuksquareNelson$ is $\BelLuksquareNelson$-valid iff $\alpha^{\Box,\lozenge}$ is $\ProbLukNelsonmodal$-valid.
\end{enumerate}
\end{theorem}
\begin{proof}
We begin with (1). Consider the ‘only if’ direction: we assume that either (i) $e_1(\alpha^\boxplus)=x<1$ in some $\ProbLukordermodal$-model or (ii) $e_1(\alpha^\boxminus)=y>0$ in some $\ProbLukordermodal$-model. We prove by induction that in this case, there are some $\BelLuksquareorder$-models where $e'_1(\alpha)=x$ or $e'_2(\alpha)=y$, respectively.

Let $\alpha^\boxplus=\Prob(\Box\phi)$ and $e_1(\Prob(\Box\phi))=x<1$. This means that $\pi(|\Prob(\Box\phi)|^+)=x$. Using Remark~\ref{rem:BDmodalprobabilitytobeliefplausibility}, we obtain that there is a belief function $\bel$ s.t.\ $\bel(|\phi|^+)=x<1$. Thus, $e^\bel_1(\Bel\phi)=x<1$ for the evaluation $e^\bel_1$ induced by $\bel$, as required. The induction steps can be obtained by a simple application of the induction hypothesis since $\ProbLukordermodal$ and $\BelLuksquareorder$ use $\Luksquareorder$ as their outer-layer logic. Thus, (i) is tackled.

For (ii), we proceed similarly. Let $\alpha^{\boxminus}=\Prob(\Box\neg\phi)=y>0$. Again, we obtain that there is a belief function $\bel$ s.t.\ $\bel(|\neg\phi|^+)=x>0$ (i.e., $\bel(|\phi|^-)=y>0$). Hence, $e^\bel_2(\Bel\phi)=x>0$ for the evaluation induced by $\bel$. The cases of propositional connectives can be tackled similarly, so, we only consider $\alpha^\boxminus=\gamma^\boxminus\ominus\beta^\boxminus$. Let $e_1(\gamma^\boxminus\ominus\beta^\boxminus)=y>0$. Thus, $z\!=\!e_1(\gamma^\boxminus)\!>\!e_1(\beta^\boxminus)\!=\!z'$ with $z\!-\!z'\!=\!y$. Applying the induction hypothesis, we have that $e'_2(\gamma)\!=\!z$, $e'_2(\beta)\!=\!z'$, and $z\!-\!z'\!=\!y$. Hence, $e'_2(\beta\rightarrow\gamma)=y>0$.

Let us now deal with the ‘if’ direction. Assume that $\alpha$ is not $\BelLuksquareorder$-valid, i.e., either (i) there is a $\BelLuksquareorder$-model where $e_1(\alpha)=x<1$ or (ii) there is a model where $e_2(\alpha)=y>0$. We prove by induction that there is a~$\ProbLukordermodal$-model s.t.\ $e'_1(\alpha^\boxplus)=x<1$ or $e'_1(\alpha^\boxminus)=y>0$.

First, if $\alpha=\Bel\phi$ and $e_1(\Bel\phi)=x<1$ in some $\BelLuksquareorder$-model, then we have that there is a~$\BD$ $\bel$-model s.t.\ $\bel(|\phi|^+)=x$. Hence, by Theorem~\ref{theorem:BDmodalprobability}, there is a~$\BD$ probabilistic Kripke model s.t.\ $\pi(|\Box\phi|^+)=x$, and thus, $e^\pi_1(\Prob(\Box\phi))=x<1$ for the induced valuation $e^\pi_1$. In the second case, we have $e_2(\Bel\phi)=y>0$ in some $\BelLuksquareorder$-model. Thus, $\bel(|\phi|^-)=y$ which is equivalent to $\bel(|\neg\phi|^+)=x$. Again, applying Theorem~\ref{theorem:BDmodalprobability}, we obtain that there is a probabilistic Kripke model with $\pi(|\Box\neg\phi|^+)=y>0$. Hence, $e^\pi_1(\Prob(\Box\neg\phi))=y>0$, as required. The cases of propositional connectives can be obtained by simple applications of the induction hypothesis as in the ‘only if’ direction.

(2) can be shown similarly. Assume that $e_1(\alpha^{\Box,\lozenge})=x<1$ in some $\ProbLukNelsonmodal$-model. Applying Remark~\ref{rem:BDmodalprobabilitytobeliefplausibility}, one can show by induction that there is a~$\BelLuksquareNelson$-model with $e'_1(\alpha)=x$. Conversely, using Theorem~\ref{theorem:BDmodalprobability}, one obtains that if $\alpha$ \emph{is not $\BelLuksquareNelson$-valid}, then $\alpha^{\Box,\lozenge}$ is not $\ProbLukNelsonmodal$-valid either.
\end{proof}

One can show that both translations are linear.
\begin{lemma}\label{lemma:lineartranslationbelief}~
\begin{enumerate}
\item Let $\alpha\in\LBelLuksquareorder$. Then $\lmc(\alpha^\boxplus)=\Omc(\lmc(\alpha))$ and $\lmc(\alpha^\boxminus)=\Omc(\lmc(\alpha))$.
\item Let $\beta\in\LBelLuksquareNelson$. Then $\lmc(\beta^{\Box,\lozenge})=\Omc(\lmc(\alpha))$.
\end{enumerate}
\end{lemma}
\begin{proof}
Analogously to Lemma~\ref{lemma:lineartranslationprobability}
\end{proof}
\subsection{Complexity}
Theorem~\ref{theorem:beliefprobabilitytwolayeredembedding} and Lemma~\ref{lemma:lineartranslationbelief} give us a~polynomial reduction from validity in $\BelLuksquareorder$ and $\BelLuksquareNelson$ to $\ProbLukordermodal$ and $\ProbLukNelsonmodal$. It is, thus, sufficient to prove that $\ProbLukordermodal$-validity and $\ProbLukNelsonmodal$-validity are $\conp$-complete. Recall that in the proof of Theorem~\ref{theorem:LukProbnpcompleteness}, we used canonical models of $\BD$ built over the powerset of all literals occurring in a~formula which we encoded via $u_v$ variables. We define canonical models for $\BD$ with $\mathbf{S5}$ modalities.
\begin{definition}[Clusters]\label{def:cluster}
Let $\mathfrak{M}=\langle W,R,v^+,v^-\rangle$ be a~$\BD$ Kripke model. A~\emph{cluster} is any subset $X\subseteq W$ closed under $R$. I.e., if $w\in X$ and $wRw'$, then $w'\in X$.
\end{definition}
\begin{definition}[Bisimilarity]\label{def:bisimilarity}
Two $\BD$ Kripke models $\mathfrak{M}=\langle W,R,v^+,v^-\rangle$ and $\mathfrak{M}'=\langle W',R',v'^+,v'^-\rangle$ are called \emph{bisimilar} (denoted $\mathfrak{M}\simeq\mathfrak{M}'$) if there is a~relation $\Zmc\subseteq W\times W'$ s.t.\ for every $p\in\Prop$, $x\in W$, and $x'\in W$, the following conditions hold:
\begin{enumerate}
\item $\mathfrak{M},x\vDash^+p$ iff $\mathfrak{M}',x'\vDash^+p$ and $\mathfrak{M},x\vDash^-p$ iff $\mathfrak{M}',x'\vDash^-p$ for every $x\in W$ and $x'\in W'$ s.t.\ $x\Zmc x'$;
\item if $wRx$ and $w\Zmc w'$, then there is $x'\in W'$ s.t.\ $w'R'x'$ and $x\Zmc x'$;
\item if $w'R'x'$ and $w\Zmc w'$, then there is $x\in X$ s.t.\ $wRx$ and $x\Zmc x'$.
\end{enumerate}

Bisimilarity of Kripke models with two relations are defined similarly but conditions (2) and (3) have to hold for both relations.
\end{definition}

The following statement is straightforward to establish.
\begin{proposition}\label{prop:bisimilarity}
Let $\mathfrak{M}\simeq\mathfrak{M}'$, $\Zmc$ be the bisimilarity relation, and $w\Zmc w'$. Then for every $\sigma$ over $\{\neg,\wedge,\vee,\Box,\lozenge\}$, it follows that
\begin{align*}
\mathfrak{M},w\vDash^+\sigma&\text{ iff }\mathfrak{M}',w'\vDash^+\sigma&\mathfrak{M},w\vDash^-\sigma&\text{ iff }\mathfrak{M}',w'\vDash^-\sigma
\end{align*}
\end{proposition}
\begin{definition}[Canonical model for $\BD$ with $\mathbf{S5}$ $\Box$]\label{def:S5BDcanonicalmodel}
Let $\Gamma=\{\Box\phi_1,\ldots\Box\phi_m\}$ and $\Lit[\Gamma]=\{l_1,\ldots,l_n\}$. The \emph{canonical model of $\Gamma$} is the disjoint union of all pairwise non-bisimilar clusters on $2^{\Lit[\Gamma]}$, i.e., the following structure:
\begin{align*}
\canonicalBDequiv_\Gamma&=\biguplus\limits_{\Smsf\subseteq2^{\Lit[\Gamma]}}\langle\Smsf,\Smsf\times\Smsf,v^+,v^-\rangle\tag{$\Smsf\neq\varnothing$}
\end{align*}
with valuations $v^+$ and $v^-$ defined as below:
\begin{align*}
w\in v^+(p)&\text{ iff }p\in w&w\in v^-(p)&\text{ iff }\neg p\in w
\end{align*}
\end{definition}

Importantly, as the next statements show, we can prove a~version of the small model property for probabilities of $\Box\phi$ and $\lozenge\phi$ formulas where $\phi\in\LBD$. Namely, we show that in arbitrary models, we only need clusters whose size is linearly bounded by the number of formulas and that we can assume the measure to be positive only on clusters with linearly bounded size in $\canonicalBDequiv$.
\begin{lemma}[Small model property]\label{lemma:SMPmeasurebox}
Let $\mathfrak{M}=\langle W,R,v^+,v^-,\pi\rangle$ be a~$\BD$ probabilistic Kripke model, $\phi_1,\ldots,\phi_m\subseteq\LBD$, $1\leq k\leq m$ and $\Gamma=\{\Box\phi_i\mid i\leq k\}\cup\{\lozenge\phi_i\mid k<i\leq m\}$. Then there is a~model $\mathfrak{M}_\smsf=\langle W_\smsf,R_\smsf,v^+_\smsf,v^-_\smsf,\pi_\smsf\rangle$ s.t.\ $W_\smsf\subseteq W$, $R_\smsf=R_{|W_\smsf}$, $v^+_\smsf$ and $v^-_\smsf$ are restrictions of $v^+$ and $v^-$ on $W_\smsf$, every cluster of $\mathfrak{M}_\smsf$ contains at most $2m+1$ states, and for every $\sigma\in\Gamma$, it holds that
\begin{align*}
\pi(|\sigma|^+)&=\pi_\smsf(|\sigma|^+)&\pi(|\sigma|^-)&=\pi_\smsf(|\sigma|^-)
\end{align*}
\end{lemma}
\begin{proof}
Observe that $\mathfrak{M}$ is a~disjoint union of all its clusters and that for every $\sigma\in\Gamma$ and every cluster $X\subseteq W$ the following two statements hold:
\begin{align*}
\text{either }\forall x\in X:\mathfrak{M},x,\vDash^+\sigma&\text{ or }\forall x\in X:\mathfrak{M},x,\nvDash^+\sigma;\\
\text{either }\forall x\in X:\mathfrak{M},x,\vDash^-\sigma&\text{ or }\forall x\in X:\mathfrak{M},x,\nvDash^-\sigma.
\end{align*}
We show how to remove ‘redundant’ \emph{states} from clusters containing more than $2m+1$ states and redefine the measure accordingly.

Now let $X$ be a~cluster in $W$ with more than $2m+1$ states. Let further, $\Phi^\forall_X,\Phi^\exists_X\subseteq\Gamma$ be such that 
\begin{align*}
\Phi^\forall_X&=\{\Box\phi\mid\forall x\in X:\mathfrak{M},x\vDash^+\Box\phi\}\cup\{\lozenge\phi\mid\forall x\in X:\mathfrak{M},x\vDash^-\lozenge\phi\}%\label{equ:SMPmeasurea}\tag{$a$}
\\
\Phi^\exists_X&=\{\Box\phi\mid\forall x\in X:\mathfrak{M},x\vDash^-\Box\phi\}\cup\{\lozenge\phi\mid\forall x\in X:\mathfrak{M},x\vDash^+\lozenge\phi\}%\label{equ:SMPmeasureb}\tag{$b$}
\end{align*}
It is clear from the semantics of $\BD$ with $\mathbf{S5}$ modalities (recall Definition~\ref{def:BDprobabilisticKripkemodel}) that no matter which and how many states we remove from $X$
\begin{enumerate}
\item all $\Box$-formulas from $\Phi^\forall_X$ will remain true in every state non-removed of~$X$ while all $\Box$-formulas from $\Gamma\setminus\Phi^\exists_X$ will remain non-false in every non-removed state of $X$;
\item dually, all $\lozenge$-formulas from $\Phi^\forall_X$ will remain false in every state non-removed of~$X$ while all $\lozenge$-formulas from $\Gamma\setminus\Phi^\exists_X$ will remain non-true in every non-removed state of $X$.
\end{enumerate}
Let us now show how to construct a~cluster $X_\smsf$ containing at most $2m+1$ states s.t.\ $\Phi^\forall_X=\Phi^\forall_{X_\smsf}$ and $\Phi^\exists_X=\Phi^\exists_{X_\smsf}$.

First, we take any state $x\in X$. Clearly, all $\Box$- and $\lozenge$-formulas from $\Phi^\forall_X$ will remain true (or, respectively, false) in $x$. Likewise, $\Box$-formulas from $\Gamma\setminus\Phi^\exists_X$ will remain non-false in $x$ and $\lozenge$-formulas from $\Gamma\setminus\Phi^\exists_X$ will remain non-true.

Now let $\Box\phi\in\Phi^\exists_X$ be a~formula s.t.\ it was \emph{false} in $X$ but became non-false in $x$. This means that there was some $x_{\Box\phi}\in X$ s.t.\ $\mathfrak{M},x_{\Box\phi}\vDash^-\phi$ but $x_{\Box\phi}\neq x$. We set $X_{\Box\phi}=\{x,x_{\Box\phi}\}$. It is clear that $\Box\phi$ is false in $X_{\Box\phi}$ but all other formulas remained true or non-false because they were true or non-false in $X$ and $X_{\Box\phi}\subseteq X$. Now we take the next formula, say, $\Box\phi'\in\Phi^\exists_X$ that was false in $X$ but became non-false in $x$. If it is still non-false in $X_{\Box\phi}$, it means that there is some $x_{\Box\phi'}\in X\setminus X_{\Box\phi}$ s.t.\ $\mathfrak{M},x_{\Box\phi'}\vDash^-\phi'$. Thus, we set $X_{\Box\phi'}=X_{\Box\phi}\cup\{x_{\Box\phi'}\}$. We repeat this procedure until we tackle all $\Box$-formulas from $\Phi^\exists_X$ that became non-false in $x$. Then we do the same with $\lozenge$-formulas from $\Phi^\exists_X$ that became \emph{non-true} in $x$. Since there are at most $m$ formulas in $\Phi^\exists_X$, we will have added at most $m$ states to $\{x\}$. Denote the resulting cluster $X_-$.

After that, we do the same with $\Box$-formulas from $\Gamma\setminus\Phi^\forall_X$ that became true in $x$ and $\lozenge$-formulas from $\Gamma\setminus\Phi^\forall_X$ that became false in $x$. This will also add at most $m$ states to $X_-$. The resulting cluster which we call $X_\smsf$ will, thus, contain at most $2m+1$ states. Moreover, we have that $\Phi^\forall_X=\Phi^\forall_{X_\smsf}$ and $\Phi^\exists_X=\Phi^\exists_{X_\smsf}$, as required.

It remains to define $\pi_\smsf$ on $X_\smsf$. As $X_\smsf\subseteq X$, let $X_\smsf=X'\cup\{x''\}$ and define
\begin{align*}
\pi_\smsf(\{x'\})&=
\begin{cases}
\pi(\{x'\})&\text{ if }x'\in X'\\
\displaystyle\sum\limits_{y\in X\setminus X_\smsf}\!\!\pi(\{y\})&\text{ if }x'=x''
\end{cases}
\end{align*}
We repeat this procedure with every cluster. It is clear that the obtained model will consist of clusters containing at most $2m+1$ states and that $\pi(|\sigma|^+)=\pi_\smsf(|\sigma|^+)$ and $\pi(|\sigma|^-)=\pi_\smsf(|\sigma|^-)$ for every $\sigma\in\Gamma$, it holds that
\begin{align*}
\pi(|\sigma|^+)&=\pi_\smsf(|\sigma|^+)&\pi(|\sigma|^-)&=\pi_\smsf(|\sigma|^-)
\end{align*}
\end{proof}

The next statement shows that any measure assignment to formulas of the form $\Box\phi$ and $\lozenge\phi$ on a~given $\BD$ probabilistic Kripke model can be transferred to a~measure assignment on $\canonicalBDequiv$.
\begin{theorem}\label{theorem:BDS5canonicalmeasure}~
\begin{enumerate}
\item Let $\phi_1,\ldots,\phi_m\in\LBD$, $\Gamma=\{\Box\phi_1,\ldots\Box\phi_{l},\lozenge\phi_{l+1},\ldots,\lozenge\phi_m\}$, and $\mathfrak{M}=\langle W,R,v^+,v^-,\pi\rangle$ be a~probabilistic Kripke model. Then there is a probability measure $\pi^\Cmsf$ on $\canonicalBDequiv_{\Gamma}$ s.t.
\begin{enumerate}
\item $\pi^\Cmsf(|\sigma|^+)=\pi(|\sigma|^+)$ and $\pi^\Cmsf(|\sigma|^-)=\pi(|\sigma|^-)$ for every $\sigma\in\Gamma$;
\item if $X$ is a~cluster in $\canonicalBDequiv_\Gamma$ s.t.\ $|X|>2m+1$, then $\pi^\Cmsf(X)=0$.
\end{enumerate}
\item Let $\{\phi_i\mid i\leq k\}\cup\{\chi_j\mid j\leq l\}\subseteq\LBD$, $k+l=m$, $k'\leq k$, $l'\leq l$, $\Gamma_1=\{\Box_1\phi_i\mid i\leq k'\}\cup\{\lozenge_1\phi_i\mid k'<i\leq k\}$, $\Gamma_2=\{\Box_2\chi_j\mid j\leq l'\}\cup\{\lozenge_2\chi_j\mid l'<j\leq l\}$, $\Gamma=\Gamma_1\cup\Gamma_2$, and $\mathfrak{M}=\langle W,R_1,R_2,v^+,v^-,\pi_1,\pi_2\rangle$ be a~probabilistic Kripke model. Then there are probability measures $\pi^\Cmsf_1$ and $\pi^\Cmsf_2$ on $\canonicalBDequiv_{\Gamma}$ s.t.
\begin{enumerate}
\item $\pi^\Cmsf_1(|\sigma|^+)=\pi_1(|\sigma|^+)$, $\pi^\Cmsf_1(|\sigma|^-)=\pi_1(|\sigma|^-)$, $\pi^\Cmsf_2(|\tau|^+)=\pi_2(|\tau|^+)$, and $\pi^\Cmsf_2(|\tau|^-)=\pi_2(|\tau|^-)$ for every $\sigma\in\Gamma_1$ and $\tau\in\Gamma_2$;
\item if $X$ is a~cluster in $\canonicalBDequiv_\Gamma$ s.t.\ $|X|>2m+1$, then $\pi^\Cmsf_1(X)=\pi^\Cmsf_2(X)=0$.
\end{enumerate}
\end{enumerate}
\end{theorem}
\begin{proof}
We begin with (1). As there are only $m$ formulas, it is clear that $\Lit[\Gamma]$ is finite. We let $|\Lit[\Gamma]|=k$. Thus, $\mathfrak{M}$ contains at most $\displaystyle\sum\limits_{j=1}^{2^k}\binom{2^k}{j}=2^{2^k}-1$ \emph{pairwise non-bisimilar} clusters (recall Definition~\ref{def:cluster} for the notion of cluster). In addition to that, using Lemma~\ref{lemma:SMPmeasurebox}, we can w.l.o.g.\ assume that these clusters contain at most $2m+1$ states. Moreover, $W$ is the disjoint union of \emph{all} clusters, whence, $\displaystyle\sum\limits_{X\text{ is a~cluster}}\!\!\!\!\!\!\!\!\pi(X)\!=\!1$, and furthermore, for every cluster $X\subseteq\mathfrak{M}$, there is a~cluster $X^\Cmsf\subseteq\canonicalBDequiv_\Gamma$ s.t.\ $X\simeq X^\Cmsf$. From here, it is clear that
\begin{align*}
\mathfrak{M},w\vDash^+\sigma&\text{ iff }\canonicalBDequiv_{\Gamma},w^\Cmsf\vDash^+\sigma&\mathfrak{M},w\vDash^-\sigma&\text{ iff }\canonicalBDequiv_{\Gamma},w^\Cmsf\vDash^-\sigma
\end{align*}
holds for every $\sigma$ when $w\Zmc w^\Cmsf$.

Now, recall from Definition~\ref{def:S5BDcanonicalmodel} that $\canonicalBDequiv_{\Gamma}$ is the disjoint union of all pairwise non-bisimilar clusters. Thus, we can define $\pi^\Cmsf$ as follows:
\begin{align*}
\pi^\Cmsf(\{w^\Cmsf\})&=\sum\limits_{\scriptsize{\begin{matrix}w\Zmc w^\Cmsf\\w\!\in\!W\end{matrix}}}\pi(\{w\})
\end{align*}
It is now easy to see that $\pi^\Cmsf(|\sigma|^+)=\pi(|\sigma|^+)$ and $\pi^\Cmsf(|\sigma|^-)=\pi(|\sigma|^-)$ for every $\sigma\in\Gamma$, as required. Moreover, since $\mathfrak{M}$ w.l.o.g.\ contains only clusters with at most $2m+1$ states, no cluster in $\canonicalBDequiv_\Gamma$ with a~greater number of states has a~positive measure assignment under $\pi^\Cmsf$.

The proof of (2) is similar. First, we take the \emph{mono-relational reducts} of $\mathfrak{M}$, i.e., the following two models: $\mathfrak{M}_{R_1}=\langle W,R_1,v^+,v^-,\pi\rangle$ and $\mathfrak{M}_{R_2}=\langle W,R_2,v^+,v^-,\pi\rangle$. Since every formula in $\Gamma$ contains only one modality and thus depends only on $R_1$ or only on $R_2$, it is clear that the following equivalences hold for every $w\in W$, $\sigma\in\Gamma_1$, and $\tau\in\Gamma_2$:
\begin{align*}
\mathfrak{M},w\vDash^+\sigma&\text{ iff }\mathfrak{M}_{R_1},w\vDash^+\sigma&\mathfrak{M},w\vDash^-\sigma&\text{ iff }\mathfrak{M}_{R_1},w\vDash^-\sigma\\
\mathfrak{M},w\vDash^+\tau&\text{ iff }\mathfrak{M}_{R_2},w\vDash^+\tau&\mathfrak{M},w\vDash^-\tau&\text{ iff }\mathfrak{M}_{R_2},w\vDash^-\tau
\end{align*}
Now, it is clear that $\mathfrak{M}_{R_1}$ and $\mathfrak{M}_{R_2}$ are disjoint unions of clusters (w.r.t.\ $R_1$ and $R_2$, respectively). We can again use Lemma~\ref{lemma:SMPmeasurebox} and assume w.l.o.g.\ that clusters w.r.t.\ $R_1$ and $R_2$ contain at most $2m+1$ states. Hence:
\begin{enumerate}
\item[(a)] for every cluster $X\subseteq\mathfrak{M}_{R_1}$, there is a~cluster $X^\Cmsf\subseteq\canonicalBDequiv_\Gamma$ s.t.\ $X\simeq X^\Cmsf$;
\item[(b)] for every cluster $Y\subseteq\mathfrak{M}_{R_2}$, there is a~cluster $Y^\Cmsf\subseteq\canonicalBDequiv_\Gamma$ s.t.\ $Y\simeq Y^\Cmsf$.
\end{enumerate}
We use $\Zmc_1$ and $\Zmc_2$ to denote bisimilarity relations that witness (a) and (b), respectively. Now, we define $\pi^\Cmsf_1$ and $\pi^\Cmsf_2$ as below.
\begin{align*}
\pi^\Cmsf_1(\{w^\Cmsf\})&=\sum\limits_{\scriptsize{\begin{matrix}w\Zmc_1 w^\Cmsf\\w\!\in\!W\end{matrix}}}\pi_1(\{w\})&\pi^\Cmsf_2(\{w^\Cmsf\})&=\sum\limits_{\scriptsize{\begin{matrix}w\Zmc_2 w^\Cmsf\\w\!\in\!W\end{matrix}}}\pi_2(\{w\})
\end{align*}
Again, it is straightforward to check that $\pi^\Cmsf_1(|\sigma|^+)=\pi_1(|\sigma|^+)$, $\pi^\Cmsf_1(|\sigma|^-)=\pi_1(|\sigma|^-)$, $\pi^\Cmsf_2(|\tau|^+)=\pi_2(|\tau|^+)$, and $\pi^\Cmsf_2(|\tau|^-)=\pi_2(|\tau|^-)$ for every $\sigma\in\Gamma_1$ and $\tau\in\Gamma_2$.
\end{proof}

The decision algorithms for $\ProbLukordermodal$ and $\ProbLukNelsonmodal$ are based upon that presented in Theorem~\ref{theorem:LukProbnpcompleteness}. Thus, we need a~tableaux calculus for $\NLuk$. It was given by~\cite{BilkovaFrittellaKozhemiachenko2021TABLEAUX} and we recall it here. All notions~--- open and closed branches, interpretations of entries, and satisfying valuations of branches --- are the same as in $\TLuksquareorder$ (cf.~Definitions~\ref{def:L2triangleconstrainttableaux} and~\ref{def:satisfyingvaluation}), so we only list the rules in the next definition.
\begin{definition}[$\TNLuk$ --- the tableau calculus for $\NLuk$]
The rules are as follows.
\begin{align*}
\neg\!\leqslant_1\!\dfrac{\neg\phi\leqslant_1i}{\phi\leqslant_2i}&&
\neg\!\leqslant_2\!\dfrac{\neg\phi\leqslant_2i}{\phi\leqslant_1i}&&
\neg\!\geqslant_1\!\dfrac{\neg\phi\geqslant_1i}{\phi\geqslant_2i}&&
\neg\!\geqslant_2\!\dfrac{\neg\phi\geqslant_2i}{\phi\geqslant_1i}\\[.5em]
\Ninvol\!\leqslant_1\!\dfrac{\Ninvol\phi\leqslant_1i}{\phi\geqslant_11-i}
&&
\Ninvol\!\leqslant_2\!\dfrac{\Ninvol\phi\leqslant_2i}{\phi\leqslant_1i}
&&
\Ninvol\!\geqslant_1\!\dfrac{\Ninvol\phi\geqslant_1i}{\phi\leqslant_11-i}
&&
\Ninvol\!\geqslant_2\!\dfrac{\Ninvol\phi\geqslant_2i}{\phi\geqslant_2i}\\[.5em]
\weakrightarrow\leqslant_1\dfrac{\phi_1\weakrightarrow\phi_2\leqslant_1i}{i\geq1\left|\begin{matrix}\phi_1\geqslant_11\!-\!i\!+\!j\\\phi_2\!\leqslant_1\!j\\j\leq i\end{matrix}\right.}
&&
\weakrightarrow\leqslant_2\dfrac{\phi_1\weakrightarrow\phi_2\leqslant_2i}{\begin{matrix}\phi_1\!\leqslant_2\!i\!+\!j\\\phi_2\!\leqslant_1\!1\!-\!j\end{matrix}}
&&
\weakrightarrow\geqslant_1\dfrac{\phi_1\weakrightarrow\phi_2\geqslant_1 i}{\begin{matrix}\phi_1\!\leqslant_1\!1\!-\!i\!+\!j\\\phi_2\!\geqslant_1\!j\end{matrix}}
&&
\weakrightarrow\geqslant_2\dfrac{\phi_1\weakrightarrow\phi_2\geqslant_2i}{i\leq0\left|\begin{matrix}\phi_1\!\geqslant_2\!i\!+\!j\\\phi_2\!\geqslant_1\!1\!-\!j\\j\leq 1-i\end{matrix}\right.}\\[.5em]
\wedge\leqslant_1\dfrac{\phi_1\wedge\phi_2\leqslant_1i}{\phi_1\leqslant_1i\mid\phi_2\leqslant_1i}
&&
\wedge\leqslant_2\dfrac{\phi_1\wedge\phi_2\leqslant_2i}{\begin{matrix}\phi_1\leqslant_2i\\\phi_2\leqslant_2i\end{matrix}}
&&
\wedge\geqslant_1\dfrac{\phi_1\wedge\phi_2\geqslant_1i}{\begin{matrix}\phi_1\geqslant_1i\\\phi_2\geqslant_1i\end{matrix}}
&&
\wedge\geqslant_2\dfrac{\phi_1\wedge\phi_2\geqslant_2i}{\phi_1\geqslant_2i\mid\phi_2\geqslant_2i}
\end{align*}
\end{definition}
\begin{remark}\label{rem:TNLuklinearisation}
Note that all connectives in $\NLuk$ are continuous and thus all rules of $\TNLuk$ can be linearised. This will require the introduction of another sort of variables that range over $\{0,1\}$. To make the presentation of both tableaux calculi uniform, however, we decided against the linear version of the tableau rules.
\end{remark}
\begin{theorem}\label{theorem:BDmodalprobabilityNP}
Validity in $\ProbLukordermodal$ and $\ProbLukNelsonmodal$ is $\conp$-complete.
\end{theorem}
\begin{proof}
As $\Luksquareorder$ and $\NLuk$ are $\conp$-complete~\cite{BilkovaFrittellaKozhemiachenko2021TABLEAUX}, we only show the upper bound. The proof is similar to that of Theorem~\ref{theorem:LukProbnpcompleteness}, so we address the main differences.

We begin with $\ProbLukordermodal$. We show how to check the falsifiability of $\alpha\in\LProbLukordermodal$ in nondeterministic polynomial time. Assume that $\alpha$ is built over $\Prob\sigma_1$, \ldots, $\Prob\sigma_n$ and that $|\Lit(\alpha)|=n$ (just as in Theorem~\ref{theorem:LukProbnpcompleteness}, we can add superfluous modal atoms) and that $\alpha$ is outer-$\neg$-free (this holds because $e(\neg\Prob\Box\phi)=e(\Prob\lozenge\neg\phi)$ and $e(\neg\Prob\lozenge\phi)=e(\Prob\Box\neg\phi)$ in every model).

We construct a~$\TLuksquareorder$ constraint tableau beginning with $\{\alpha^-\leqslant_1c,c<1\}$ where $\alpha^-$ is the result of replacement of every modal atom $\Prob\sigma_i$ with a~fresh variable $q_{\sigma_i}$ and check in nondeterministic polynomial time whether it is closed. If it has open branches, we guess one (say, $\Bmc$) and consider its corresponding system of linear inequalities.
\begin{align}
\tag{$\mathrm{LI}(1)_\Box^\Bmc$}\label{equ:MIP1Box}
z_1\triangledown t_1,\ldots,z_n\triangledown t_n,k_1\leq k'_1,\ldots,k_r\leq k'_r,m_1\geq1,\ldots,m_s\geq1,m'_1\leq0,\ldots,m'_t\leq0
\end{align}
Here, $z_i$'s correspond to the values of $q_{\sigma_i}$'s in $\alpha^-$, and $t_i$'s are linear polynomials that label~$q_{\sigma_i}$'s. Numerical constraints give us $k$'s, $k'$'s, $m$'s and $m'$'s. Denote the number of inequalities and the number of variables in~\ref{equ:MIP1Box} with $l_1$ and $l_2$, respectively. It is clear that $l_1=\Omc(\lmc(\alpha^-))$ and $l_2=\Omc(\lmc(\alpha^-))$.

To check that $z_i$'s are coherent as probabilities of $\sigma_i$'s, we introduce $2^{2^n}-1$ new variables of the form $u_v$ with $v\in\{0,1,;\}^*$. Here, $v$ encodes the valuation of literals of $\alpha$ in the corresponding cluster of $\canonicalBDequiv_{\alpha}$ (cf.~Definitions~\ref{def:cluster} and~\ref{def:S5BDcanonicalmodel}). For example, if $\Lit(\alpha)=\{p_1,\neg p_1,p_2,\neg p_2\}$, $u_{1000;1100;0101}$ encodes $\{\{p_1\},\{p_1,\neg p_1\},\{\neg p_1,\neg p_2\}\}$. Additionally, we put $a_{i,v}=1$ if $\canonicalBDequiv_{\alpha},w\vDash^+\sigma_i$ for every state $w$ in the cluster corresponding to $v$ and $a_{i,v}=0$, otherwise. We then add the following equalities for every $i$.
\begin{align}
\label{equ:MIP2expBox}\tag{$\mathrm{LI}(2\exp)_\Box^\Bmc$}
\sum\limits_{v}u_v&=1&\sum\limits_{v}(a_{i,v}\cdot u_v)&=z_i
\end{align}

As in Theorem~\ref{theorem:LukProbnpcompleteness}, we guess a~list $L$ of at most $l_1+n+1$ words $v$. By Theorem~\ref{theorem:BDS5canonicalmeasure}, we w.l.o.g.\ assume that only clusters with at most $2n+1$ states have a~positive measure assignment. Thus, $L$ contains only $u_v$'s whose length is not greater than $n\cdot(2n+1)$, whence, the size of $L$ is not greater than $n\cdot(2n+1)\cdot(l_1+n+1)$. We compute the values of $a_{i,v}$'s for $i\leq n$ and $v\in L$ which takes deterministic polynomial time w.r.t.\ $\lmc(\alpha^-)$ because clusters contain at most $2n+1$ states. This gives us the following system of inequalities.
\begin{align}
\label{equ:MIP2polyBox}\tag{$\mathrm{LI}(2\mathrm{poly})_\Box^\Bmc$}
\sum\limits_{v\in L}u_v&=1&\sum\limits_{v\in L}(a_{i,v}\cdot u_v)&=z_i
\end{align}
One can see that the size of \ref{equ:MIP1Box}$\cup$\ref{equ:MIP2polyBox} is polynomial in $\lmc(\alpha^-)$ as there are polynomially many (in)equalities containing polynomially many variables with linearly bounded labels. Hence, it can be solved in nondeterministic polynomial time. The rest of the proof follows that of Theorem~\ref{theorem:LukProbnpcompleteness}.

The $\conp$-membership proof for $\ProbLukNelsonmodal$ is similar, so, we provide only a~sketch. The main differences are as follows. First, there are two modalities --- $\Prob_1$ and $\Prob_2$; this is why, when transforming $\alpha$ into $\alpha^-$, we replace $\Prob_1\sigma_i$'s with $q_{\sigma_{1_i}}$'s and $\Prob_2\sigma_i$'s with $q_{\sigma_{2_i}}$'s. Thus, we introduce two sorts of $z$'s ($z_{1_i}$'s and $z_{2_i}$'s) in the system of inequalities corresponding to an open branch in the $\TNLuk$ tableau for~$\alpha^-$. Note, however, that since $\Box_1$ and $\lozenge_1$ do not occur in the same modal atom as $\Box_2$ or $\lozenge_2$, we can utilise $\canonicalBDequiv_{\alpha^-}$ for both sets of modalities. However (recall Theorem~\ref{theorem:BDS5canonicalmeasure}), we will need two independent measures.

Second, we need two sorts of $u_v$'s~--- $u^1_v$ and $u^2_v$ ($v$'s are defined as in the case of $\ProbLukordermodal$)~--- to differentiate between two measures on $\canonicalBDequiv_{\alpha}$. This means that \eqref{equ:MIP2expBox} will have the following form:
\begin{align*}
\sum\limits_{v}u^1_v&=1&\sum\limits_{v}(a_{i_1,v}\cdot u^1_v)&=z_{i_1}&\sum\limits_{v}u^2_v&=1&\sum\limits_{v}(a_{i_2,v}\cdot u^2_v)&=z_{i_2}
\end{align*}
Thus, we have $l_1+2n+2$ inequalities in total. Hence, we can guess the list $L$ containing $l_1+2n+2$ variables. By Theorem~\ref{theorem:BDS5canonicalmeasure}, we need only $v$'s of length at most $n\cdot(2n+1)$. Therefore, the size of~$L$ is at most $n\cdot(2n+1)\cdot(l_1+2n+2)$ which means that it takes nondeterministic polynomial time to solve the system of inequalities obtained after guessing $L$. The result follows.
\end{proof}
\begin{corollary}\label{cor:LuksquareBeliefNP}
Validity in $\BelLuksquareorder$ and $\BelLuksquareNelson$ is $\conp$-complete.
\end{corollary}
\begin{proof}
Immediately from Lemma~\ref{lemma:ProbLukordermodalconflation} and Theorems~\ref{theorem:beliefprobabilitytwolayeredembedding} and~\ref{theorem:BDmodalprobabilityNP}.
\end{proof}
\section{Conclusion\label{sec:conclusion}}
In this paper, we considered several two-layered logics for reasoning with probability measures and belief functions defined in the $\BD$-framework. In particular, we established that the logics of $\four$- and $\pm$-probabilities can be faithfully embedded into one another (Theorems~\ref{theorem:embeddings1} and~\ref{theorem:embeddings2}) and constructed a~complete Hilbert axiomatisation for the logic of $\four$-probabilities (Theorem~\ref{theorem:completeness}). We also established $\conp$-completeness for logics of $\pm$- and $\four$-probabilities and $\conp$-completeness for logics of belief and plausibility functions in $\BD$ (Theorem~\ref{theorem:LukProbnpcompleteness} and Corollary~\ref{cor:LuksquareBeliefNP}) utilising the connection between belief assignments of propositional and modal formulas (Theorem~\ref{theorem:BDmodalprobability}) and a~version of the small model property for the canonical model (Theorem~\ref{theorem:BDS5canonicalmeasure}). Still, several important questions remain open.

First, our belief and plausibility functions are close to $\pm$-probabilities because they assign each statement $\phi$ two measures: the measure of its positive extension and the measure of its negative extension. It is thus instructive to define $\four$-valued belief and plausibility functions that will assign measures to $\phi$ according to the pure belief, pure disbelief, conflict and uncertainty extensions (cf.~Convention~\ref{conv:4measures}). This is not a trivial task since belief and plausibility functions are not additive, whence, the measures of these extensions cannot be directly obtained from $\bel(|\phi|^+)$ and $\bel(|\phi|^-)$.

Third, observe that axiomatisations of $\BelLuksquareorder$ and $\BelLuksquareNelson$ proposed by~\cite{BilkovaFrittellaKozhemiachenkoMajerNazari2023APAL} are \emph{infinite}. In the case of belief functions over the classical logic, it is shown by~\cite{GodoHajekEesteva2001IJCAI,GodoHajekEsteva2003} how to produce a~\emph{finite} axiomatisation for the two-layered logics of belief functions in the classical case using the representation of $\bel(\|\phi\|)$ as $\pi(\|\Box\phi\|)$. Recall that a similar property (Theorem~\ref{theorem:BDmodalprobability}) holds for the $\BD$ as well. It makes sense to use it to obtain a finite axiomatisation of belief and plausibility functions over $\BD$.

Finally, \cite{BilkovaFrittellaKozhemiachenkoMajer2023IJAR} proposed two-layered logics for \emph{qualitative reasoning} about different uncertainty measures. In particular, we constructed logics for qualitative reasoning with belief functions and probabilities. These logics utilised the bi-G\"{o}del logic ($\mathsf{biG}$) in the outer layer. Since $\mathsf{biG}$ is also $\np$-complete, the question arises whether we can apply the technique of~\cite{HajekTulipani2001} to prove the $\np$-completeness of two-layered logics for qualitative reasoning about uncertainty.

\begin{acknowledgements}
The first and the last author were supported by the grant \textnumero22-23022L CELIA of Grantová Agentura České Republiky. The second and third authors were supported by the grant ANR-19-CE48-0006 of Agence Nationale de la Recherche. The third author was, in addition, supported by the grant ANR-19-CHIA-0014 of Agence Nationale de la Recherche. The paper is a~part of the MOSAIC project financed by the European Union's Marie Sk\l{}odowska-Curie grant \textnumero101007627.

The authors also wish to thank the editors and the reviewers for their comments that greatly enhanced the paper.
\end{acknowledgements}
\confl~No conflict of interest to declare.
\bibliographystyle{msclike}
\bibliography{references}

\begin{thebibliography}{}

\bibitem[Arieli and Avron, 2017]{ArieliAvron2017}
{\bf Arieli, O.} \textbf{and} {\bf Avron, A.} 2017.
\newblock Four-valued paradefinite logics.
\newblock {\em Studia Logica}, 105(6):1087--1122.

\bibitem[Baaz, 1996]{Baaz1996}
{\bf Baaz, M.} 1996.
\newblock Infinite-valued {G}{\"o}del logics with $0$-$1$-projections and relativizations.
\newblock In {\em G{\"o}del'96: Logical foundations of mathematics, computer science and physics---Kurt G{\"o}del's legacy, Brno, Czech Republic, August 1996, proceedings}, volume~6, pp. 23--34. Association for Symbolic Logic.

\bibitem[Baldi et~al., 2020]{BaldiCintulaNoguera2020}
{\bf Baldi, P.}, {\bf Cintula, P.}, \textbf{and} {\bf Noguera, C.} 2020.
\newblock Classical and fuzzy two-layered modal logics for uncertainty: Translations and proof-theory.
\newblock {\em International Journal of Computational Intelligence Systems}, 13(1):988--1001.

\bibitem[Belnap, 1977]{Belnap1977fourvalued}
{\bf Belnap, N.} 1977.
\newblock {A Useful Four-Valued Logic}.
\newblock In {\bf Dunn, J.~M.} \textbf{and} {\bf Epstein, G.}, editors, {\em Modern Uses of Multiple-Valued Logic}, pp. 5--37, Dordrecht. Springer Netherlands.

\bibitem[Belnap, 2019]{Belnap2019}
{\bf Belnap, N.} 2019.
\newblock How a computer should think.
\newblock In {\bf Omori, H.} \textbf{and} {\bf Wansing, H.}, editors, {\em New Essays on Belnap-Dunn Logic}, volume 418 of {\em Synthese Library (Studies in Epistemology, Logic, Methodology, and Philosophy of Science)}. Springer, Cham.

\bibitem[B{\'i}lkov{\'a} et~al., 2021]{BilkovaFrittellaKozhemiachenko2021TABLEAUX}
{\bf B{\'i}lkov{\'a}, M.}, {\bf Frittella, S.}, \textbf{and} {\bf Kozhemiachenko, D.} 2021.
\newblock Constraint tableaux for two-dimensional fuzzy logics.
\newblock In {\bf Das, A.} \textbf{and} {\bf Negri, S.}, editors, {\em Automated Reasoning with Analytic Tableaux and Related Methods}, volume 12842 of {\em Lecture Notes in Computer Science}, pp. 20--37. Springer International Publishing.

\bibitem[B{\'\i}lkov{\'a} et~al., 2023a]{BilkovaFrittellaKozhemiachenkoMajer2023IJAR}
{\bf B{\'\i}lkov{\'a}, M.}, {\bf Frittella, S.}, {\bf Kozhemiachenko, D.}, \textbf{and} {\bf Majer, O.} 2023a.
\newblock Qualitative reasoning in a two-layered framework.
\newblock {\em International Journal of Approximate Reasoning}, 154:84--108.

\bibitem[B{\'\i}lkov{\'a} et~al., 2023b]{BilkovaFrittellaKozhemiachenkoMajer2023WoLLIC}
{\bf B{\'\i}lkov{\'a}, M.}, {\bf Frittella, S.}, {\bf Kozhemiachenko, D.}, \textbf{and} {\bf Majer, O.} 2023b.
\newblock Two-layered logics for paraconsistent probabilities.
\newblock In {\bf Hansen, H.}, {\bf Scedrov, A.}, \textbf{and} {\bf de~Queiroz, R.}, editors, {\em Logic, Language, Information, and Computation. WoLLIC 2023}, volume 13923 of {\em Lecture notes in computer science}, pp. 101--117. Springer Nature Switzerland, Cham.

\bibitem[B\'{\i}lkov\'a et~al., 2023]{BilkovaFrittellaKozhemiachenkoMajerManoorkar2023}
{\bf B\'{\i}lkov\'a, M.}, {\bf Frittella, S.}, {\bf Kozhemiachenko, D.}, {\bf Majer, O.}, \textbf{and} {\bf Manoorkar, K.} 2023.
\newblock {Describing and quantifying contradiction between pieces of evidence via Belnap Dunn logic and Dempster-Shafer theory}.
\newblock In {\bf Miranda, E.}, {\bf Montes, I.}, {\bf Quaeghebeur, E.}, \textbf{and} {\bf Vantaggi, B.}, editors, {\em Proceedings of the Thirteenth International Symposium on Imprecise Probability: Theories and Applications}, volume 215 of {\em Proceedings of Machine Learning Research}, pp. 37--47. PMLR.

\bibitem[B{\'\i}lkov{\'a} et~al., 2023]{BilkovaFrittellaKozhemiachenkoMajerNazari2023APAL}
{\bf B{\'\i}lkov{\'a}, M.}, {\bf Frittella, S.}, {\bf Kozhemiachenko, D.}, {\bf Majer, O.}, \textbf{and} {\bf Nazari, S.} 2023.
\newblock Reasoning with belief functions over {Belnap--Dunn} logic.
\newblock {\em Annals of Pure and Applied Logic},  103338.

\bibitem[B{\'\i}lkov{\'a} et~al., 2020]{BilkovaFrittellaMajerNazari2020}
{\bf B{\'\i}lkov{\'a}, M.}, {\bf Frittella, S.}, {\bf Majer, O.}, \textbf{and} {\bf Nazari, S.} 2020.
\newblock Belief based on inconsistent information.
\newblock In {\em International Workshop on Dynamic Logic}, pp. 68--86. Springer.

\bibitem[Bueno-Soler and Carnielli, 2016]{Bueno-SolerCarnielli2016}
{\bf Bueno-Soler, J.} \textbf{and} {\bf Carnielli, W.} 2016.
\newblock {Paraconsistent probabilities: Consistency, contradictions and Bayes' theorem}.
\newblock {\em Entropy (Basel)}, 18(9):325.

\bibitem[B\v{e}hounek et~al., 2011]{BehounekCintulaHajek2011MFL1}
{\bf B\v{e}hounek, L.}, {\bf Cintula, P.}, \textbf{and} {\bf H\'{a}jek, P.} 2011.
\newblock {Introduction to Mathematical Fuzzy Logic}.
\newblock In {\bf Cintula, P.}, {\bf H{\'{a}}jek, P.}, \textbf{and} {\bf Noguera, C.}, editors, {\em {Handbook of Mathematical Fuzzy Logic}}, volume~37 of {\em Studies in logic}, pp. 1--102. College Publications.

\bibitem[Cintula and Noguera, 2014]{CintulaNoguera2014}
{\bf Cintula, P.} \textbf{and} {\bf Noguera, C.} 2014.
\newblock {Modal Logics of Uncertainty with Two-Layer Syntax: {A} General Completeness Theorem}.
\newblock In {\em Proceedings of WoLLIC 2014}, pp. 124--136.

\bibitem[Dautovi{\'c} et~al., 2021]{DautovicDoderOgnjanovic2021}
{\bf Dautovi{\'c}, {\v S}.}, {\bf Doder, D.}, \textbf{and} {\bf Ognjanovi{\'c}, Z.} 2021.
\newblock An epistemic probabilistic logic with conditional probabilities.
\newblock In {\em Logics in Artificial Intelligence. JELIA 2021}, volume 12678 of {\em Lecture notes in computer science}, pp. 279--293. Springer International Publishing, Cham.

\bibitem[Delgrande and Renne, 2015]{DelgrandeRenne2015}
{\bf Delgrande, J.} \textbf{and} {\bf Renne, B.} 2015.
\newblock The logic of qualitative probability.
\newblock In {\em Twenty-Fourth International Joint Conference on Artificial Intelligence}.

\bibitem[Delgrande et~al., 2019]{DelgrandeRenneSack2019}
{\bf Delgrande, J.}, {\bf Renne, B.}, \textbf{and} {\bf Sack, J.} 2019.
\newblock The logic of qualitative probability.
\newblock {\em Artificial Intelligence}, 275:457--486.

\bibitem[Dubois, 2008]{Dubois2008}
{\bf Dubois, D.} 2008.
\newblock On ignorance and contradiction considered as truth-values.
\newblock {\em Logic Journal of the IGPL}, 16(2):195--216.

\bibitem[Dunn, 1976]{Dunn1976}
{\bf Dunn, J.} 1976.
\newblock Intuitive semantics for first-degree entailments and ‘coupled trees’.
\newblock {\em Philosophical Studies}, 29(3):149--168.

\bibitem[Dunn, 2010]{Dunn2010}
{\bf Dunn, J.} 2010.
\newblock {Contradictory information: Too much of a good thing}.
\newblock {\em Journal of Philosophical Logic}, 39:425–--452.

\bibitem[Fagin and Halpern, 1991]{FaginHalpern1991ComputationalIntelligence}
{\bf Fagin, R.} \textbf{and} {\bf Halpern, J.} 1991.
\newblock Uncertainty, belief, and probability.
\newblock {\em Computational Intelligence}, 7(3):160--173.

\bibitem[Fagin et~al., 1990]{FaginHalpernMegiddo1990}
{\bf Fagin, R.}, {\bf Halpern, J.~Y.}, \textbf{and} {\bf Megiddo, N.} 1990.
\newblock A logic for reasoning about probabilities.
\newblock {\em Information and Computation}, 87:78--128.

\bibitem[Flaminio et~al., 2013]{FlaminioGodoMarchioni2013}
{\bf Flaminio, T.}, {\bf Godo, L.}, \textbf{and} {\bf Marchioni, E.} 2013.
\newblock {Logics for belief functions on MV-algebras}.
\newblock {\em International Journal of Approximate Reasoning}, 54(4):491--512.

\bibitem[Flaminio et~al., 2022]{FlaminioGodoUgolini2022}
{\bf Flaminio, T.}, {\bf Godo, L.}, \textbf{and} {\bf Ugolini, S.} 2022.
\newblock {An Approach to Inconsistency-Tolerant Reasoning About Probability Based on {\L}ukasiewicz Logic}.
\newblock In {\bf Dupin~de Saint-Cyr, F.}, {\bf Öztürk Escoffier, M.}, \textbf{and} {\bf Potyka, N.}, editors, {\em International Conference on Scalable Uncertainty Management}, volume 13562 of {\em Lecture Notes in Artificial Intelligence}, pp. 124--138. Springer.

\bibitem[G{\"a}rdenfors, 1975]{Gardenfors1975}
{\bf G{\"a}rdenfors, P.} 1975.
\newblock Qualitative probability as an intensional logic.
\newblock {\em Journal of Philosophical Logic}, pp. 171--185.

\bibitem[Godo et~al., 2001]{GodoHajekEesteva2001IJCAI}
{\bf Godo, L.}, {\bf H{\'{a}}jek, P.}, \textbf{and} {\bf Esteva, F.} 2001.
\newblock {A Fuzzy Modal Logic for Belief Functions}.
\newblock In {\bf Nebel, N.}, editor, {\em Proceedings of the Seventeenth International Joint Conference on Artificial Intelligence, {IJCAI} 2001}, pp. 723--732. Morgan Kaufmann.

\bibitem[Godo et~al., 2003]{GodoHajekEsteva2003}
{\bf Godo, L.}, {\bf H{\'a}jek, P.}, \textbf{and} {\bf Esteva, F.} 2003.
\newblock A fuzzy modal logic for belief functions.
\newblock {\em Fundamenta Informaticae}, 57(2--4):127--146.

\bibitem[H{\"a}hnle, 1992]{Haehnle1992}
{\bf H{\"a}hnle, R.} 1992.
\newblock A new translation from deduction into integer programming.
\newblock In {\em International Conference on Artificial Intelligence and Symbolic Mathematical Computing}, pp. 262--275. Springer.

\bibitem[H{\"a}hnle, 1994]{Haehnle1994}
{\bf H{\"a}hnle, R.} 1994.
\newblock Many-valued logic and mixed integer programming.
\newblock {\em Annals of mathematics and Artificial Intelligence}, 12(3-4):231--263.

\bibitem[H{\'a}jek, 1996]{Hajek1996}
{\bf H{\'a}jek, P.} 1996.
\newblock {Getting belief functions from Kripke models}.
\newblock {\em International Journal of General Systems}, 24(3):325--327.

\bibitem[H{\'{a}}jek, 1998]{Hajek1998}
{\bf H{\'{a}}jek, P.} 1998.
\newblock {\em Metamathematics of Fuzzy Logic}, volume~4 of {\em Trends in Logic}.
\newblock Kluwer.

\bibitem[H{\'{a}}jek et~al., 1995]{HajekGodoEsteva1995}
{\bf H{\'{a}}jek, P.}, {\bf Godo, L.}, \textbf{and} {\bf Esteva, F.} 1995.
\newblock Fuzzy logic and probability.
\newblock In {\bf Besnard, P.} \textbf{and} {\bf Hanks, S.}, editors, {\em {UAI} '95: Proceedings of the Eleventh Annual Conference on Uncertainty in Artificial Intelligence, Montreal, Quebec, Canada, August 18-20, 1995}, pp. 237--244. Morgan Kaufmann.

\bibitem[H{\'a}jek and Tulipani, 2001]{HajekTulipani2001}
{\bf H{\'a}jek, P.} \textbf{and} {\bf Tulipani, S.} 2001.
\newblock Complexity of fuzzy probability logics.
\newblock {\em Fundamenta Informaticae}, 45(3):207--213.

\bibitem[Klein et~al., 2021]{KleinMajerRad2021}
{\bf Klein, D.}, {\bf Majer, O.}, \textbf{and} {\bf Rafiee~Rad, S.} 2021.
\newblock Probabilities with gaps and gluts.
\newblock {\em Journal of Philosophical Logic}, 50(5):1107--1141.

\bibitem[Mares, 1997]{Mares1997}
{\bf Mares, E.} 1997.
\newblock Paraconsistent probability theory and paraconsistent bayesianism.
\newblock {\em Logique et analyse}, pp. 375--384.

\bibitem[Metcalfe et~al., 2008]{MetcalfeOlivettiGabbay2008}
{\bf Metcalfe, G.}, {\bf Olivetti, N.}, \textbf{and} {\bf Gabbay, D.} 2008.
\newblock {\em {Proof Theory for Fuzzy Logics}}.
\newblock Applied Logic Series 36. Springer.

\bibitem[Nelson, 1949]{Nelson1949}
{\bf Nelson, D.} 1949.
\newblock Constructible falsity.
\newblock {\em The Journal of Symbolic Logic}, 14(1):16--26.

\bibitem[Odintsov and Wansing, 2017]{OdintsovWansing2017}
{\bf Odintsov, S.} \textbf{and} {\bf Wansing, H.} 2017.
\newblock {Disentangling FDE-Based Paraconsistent Modal Logics}.
\newblock {\em Studia Logica}, 105(6):1221--1254.

\bibitem[Omori and Wansing, 2017]{OmoriWansing2017}
{\bf Omori, H.} \textbf{and} {\bf Wansing, H.} 2017.
\newblock {40 years of FDE: An introductory overview}.
\newblock {\em Studia Logica}, 105(6):1021--1049.

\bibitem[Rodrigues et~al., 2021]{RodriguesBueno-SolerCarnielli2021}
{\bf Rodrigues, A.}, {\bf Bueno-Soler, J.}, \textbf{and} {\bf Carnielli, W.} 2021.
\newblock Measuring evidence: a probabilistic approach to an extension of {Belnap--Dunn} logic.
\newblock {\em Synthese}, 198(S22):5451--5480.

\bibitem[Rodriguez et~al., 2022]{RodriguezTuytEstevaGodo2022}
{\bf Rodriguez, R.}, {\bf Tuyt, O.}, {\bf Esteva, F.}, \textbf{and} {\bf Godo, L.} 2022.
\newblock {Simplified Kripke semantics for K45-like G{\"o}del modal logics and its axiomatic extensions}.
\newblock {\em Studia Logica}, 110(4):1081--1114.

\bibitem[Shafer, 1976]{Shafer1976}
{\bf Shafer, G.} 1976.
\newblock {\em A mathematical theory of evidence}.
\newblock Princeton university press.

\bibitem[Zhou, 2013]{Zhou2013}
{\bf Zhou, C.} 2013.
\newblock Belief functions on distributive lattices.
\newblock {\em Artificial Intelligence}, 201:1--31.

\end{thebibliography}
\end{document}